\documentclass[9pt]{article}
\usepackage[backref,colorlinks,linkcolor=red,anchorcolor=green,citecolor=blue]{hyperref}
\usepackage{graphicx}
\usepackage{cite}
\usepackage{amsmath,amsfonts,amssymb,amsthm}
\usepackage{mathtools}
\usepackage{color}
\usepackage[]{graphics}
\usepackage{verbatim}
\usepackage{eepic,epic}
\usepackage{epsfig,subfigure,epstopdf}
\usepackage{url}
\usepackage{leftidx} 
\usepackage{mathrsfs}
\usepackage{indentfirst}
\usepackage{enumerate}
\usepackage{algorithm,algorithmic}
\usepackage{colonequals} 
\graphicspath{{./pics/}}

\newtheorem{theorem}{Theorem}[section]
\newtheorem{lemma}{Lemma}[section]
\theoremstyle{remark}
\newtheorem{remark}{Remark}[section]
\newtheorem{definition}{Definition}[section]
\newtheorem{assumption}{Assumption}[section]
\newtheorem{example}{Example}[section]
\newtheorem{corollary}{Corollary}[section]

\allowdisplaybreaks
\usepackage{enumerate}
\newtheorem{scheme}[theorem]{Scheme}

\def\div{\mathop{\mathrm{div}}\nolimits}
\def\!{\mathop{\mathrm{!}}}
\def\d{\,\mathrm{d}}

\usepackage{color, colortbl}
\def\R{\mathbb{ R}}

\def\P{\mathcal{P}}
\def\E{\mathcal{E}}
\def\S{\mathcal{S}}

\def\F{\mathcal{F}}

\def\H{\mathcal{H}}

   \allowdisplaybreaks
\begin{document}
 \title{Wasserstein Gradient Flow Formulation of the Time-Fractional Fokker-Planck Equation\thanks{Received date, and accepted date (The correct dates will be entered by the editor).}}


          \author{Manh Hong Duong\thanks{School of Mathematics, University of Birmingham, Birmingham B15 2TT, UK (\texttt{h.duong@bham.ac.uk}).}
          \and Bangti Jin \thanks{Department of Computer Science, University College London, Gower Street, London, WC1E 6BT, UK (\texttt{b.jin@ucl.ac.uk,bangti.jin@gmail.com}).}}

         \maketitle

          \begin{abstract}
In this work, we investigate a variational formulation for a time-fractional Fokker-Planck
equation which arises in the study of complex physical systems involving anomalously slow diffusion. The model
involves a fractional-order Caputo derivative in time, and thus inherently nonlocal. The study follows
the Wasserstein gradient flow approach pioneered by \cite{JordanKinderlehrerOtto:1998}. We
propose a JKO type scheme for discretizing the model, using the L1 scheme for the Caputo fractional
derivative in time, and establish the convergence of the scheme as the time step size tends to zero.
Illustrative numerical results in one- and two-dimensional problems are also presented to show the approach.
\end{abstract}

%
%
%


\section{Introduction}
In this work, we are interested in the following time-fractional Fokker-Planck equation (FPE):
\begin{equation}\label{eq:fractionalFP}
  \left\{\begin{aligned}
    \partial^\alpha_t \rho & =\div(\nabla\Psi\rho)+\Delta\rho,\quad \mbox{in } \mathbb{R}^d\\
     \rho(0)&=\rho_0,
  \end{aligned}\right.
\end{equation}
where $\rho_0$ is the initial datum, and $\Psi$ is the forcing term.
Here, the notation $\partial_t^\alpha \varphi(t)$ denotes the Caputo fractional derivative of order
$\alpha\in(0,1)$ in time, defined by \cite[p. 91]{KilbasSrivastavaTrujillo:2006}
\begin{equation*}
  \partial_t^\alpha \varphi(t) = \frac{1}{\Gamma(\alpha)}\int_0^t(t-s)^{\alpha-1}\varphi'(s)\d s
\end{equation*}
where $\Gamma(z)$ is the Gamma function defined by $\Gamma(z)=\int_0^\infty s^{z-1}e^{-s}\d s$.
The fractional derivative $\partial_t^\alpha \varphi(t)$ recovers the usual first-order derivative $\varphi'(t)$
as $\alpha\to 1^-$ for suitably smooth functions. Therefore, the model \eqref{eq:fractionalFP} can be regarded as
a time-fractional analogue of the classical FPE.

The interest in the model \eqref{eq:fractionalFP} is motivated by an explosively growing list
of practical applications involving anomalously slow diffusion processes (a.k.a. subdiffusion), which deviate from the
classical diffusive behavior. The so-called subdiffusive process displays local motion occasionally
interrupted by long sojourns and trapping effects, and it has been widely accepted to better describe
transport phenomena in a number of practical applications in physics, biology and finance, e.g., the
study of volatility of financial markets, bacterial motion and bird flight, etc. (see the review
\cite{MetzlerKlafter:2000} for an extensive list with physical modelings). Model \eqref{eq:fractionalFP}
can be viewed as the macroscopic limit of continuous time random walk with a heavy-tailed waiting time
distribution (with a divergent mean) between consecutive jumps \cite{BarkaiMetzlerKlafter:2000}, in analogy with Brownian motion for
normal diffusion. The evolution of the probability density function (PDF) associated with the subdiffusion
process is governed by a time-fractional FPE, i.e., a FPE involving a fractional derivative in time,
as given in \eqref{eq:fractionalFP}. In the literature, there are also several works using fractional Laplacian
to describe anomalously fast diffusion processes (i.e., superdiffusion), which leads to space fractional Fokker-Planck
equations; see, e.g., \cite{SanchezCesbron2019,Bowles2015,DuongLu2019} and references therein.

There have been several important studies on the model \eqref{eq:fractionalFP} from various different perspectives
\cite{BarkaiMetzlerKlafter:2000,HenryLanglandsStraka:2010,BaeumerStraka:2017,NaneNi:2016,MagdziarzGajda:2014,
AngstmannDonnellyHenry:2015,CamilliDeMaio:2017,LiLiu:2019,Kemppainen2019}. The physical modeling using time-fractional
FPE has a long history; see \cite{MetzlerKlafter:2000} for in-depth detailed discussions.
Barkai et al \cite{BarkaiMetzlerKlafter:2000} derived the model \eqref{eq:fractionalFP} from the continuous
time random walk model in order to describe anomalous diffusion in a time-independent external force field; see
\cite{HenryLanglandsStraka:2010,AngstmannDonnellyHenry:2015} for an extension to space- and
time-dependent forcing. The well-posedness of the problem was discussed in \cite{BaeumerStraka:2017},
and the stochastic representation of the solutions was studied in \cite{MagdziarzGajda:2014,NaneNi:2016}.
Le et al \cite{LeMcLeanMustapha:2016} studied the numerical solution of the time-fractional FPE using the Galerkin
finite element method.
Camilli and De Maio \cite{CamilliDeMaio:2017} established the existence and uniqueness of a time-fractional
mean field games system. Kemppainen and Zacher \cite{Kemppainen2019} investigated the long time
behavior of a general class of nonlocal-in-time FPEs via an entropy argument, which is substantially different
from that for the classical FPE. Li and Liu \cite{LiLiu:2019}
described a discretization scheme for time-fractional gradient flow. However, none of these works has
treated the gradient flow formulation for time-fractional FPE, which was recently pointed
out by Kemppainen and Zacher \cite{Kemppainen2019} as ``\textit{an analogue of the
celebrated theorem of Jordan, Kinderlehrer and Otto on the gradient flow structure of the classical FPE
in the Wasserstein space $\mathcal{P}_2(\mathbb{R}^d)$ seems to be unknown for equation \eqref{eq:fractionalFP}
and would be highly desirable.}''

The goal of this work is to discuss the time discretization of the model \eqref{eq:fractionalFP} via a
JKO type scheme, thereby filling in an important missing piece on the time-fractional FPE pointed out
by Kemppainen and Zacher \cite{Kemppainen2019}. This is carried out following the pioneering work of Jordan,
Kinderlehrer and Otto \cite{JordanKinderlehrerOtto:1998} using the Wasserstein gradient flow for the
classical FPE. Specifically, with a time step size $\tau$, the scheme reads: given the initial datum
$\rho^0$, find $\rho^n,\ n=1,2,\ldots,N$ by minimizing
\begin{equation}\label{eqn:JKO-intro}
   \frac{C_\alpha}{2\tau^\alpha}W_2^2(\rho,\overline{\rho}^{n-1}) + \F(\rho),
\end{equation}
over the Wasserstein space $\mathcal{P}_2(\mathbb{R}^d)$, where $W_2(\cdot,\cdot)$ denotes the Wasserstein distance,
$\overline{\rho}^{n-1}$ is a convex combination of of $\rho^0,\ldots,\rho^{n-1}$ (with weights depending on the
numerical approximation of the fractional derivative $\partial_t^\alpha \rho$), $C_\alpha=1/\Gamma(2-\alpha)$ is
a fixed constant and $\F(\rho)$ is the free energy; See Section \ref{sec:JKO} for details. The term $\overline{\rho}^{k-1}$
captures the nonlocal nature / memory effect of the mathematical model. The scheme recovers the classical JKO
scheme \cite{JordanKinderlehrerOtto:1998} as $\alpha\to1^-$, and thus it represents a fractional analogue of
the latter. Numerically, it has comparable computational complexity as the classical JKO scheme, except the extra
computation of the convex combination $\overline{\rho}^{n-1}$. The main result is given in Theorem \ref{thm: main theorem}, which shows
that the piecewise constant interpolation converges weakly in $L^1((0,T) \times \mathbb{R}^d)$ to a weak solution of the
model \eqref{eq:fractionalFP}. Further, we numerically illustrate the performance of the approach, using recently
developed powerful solvers for minimization problems involving Wassserstein distance based on entropy regularization
\cite{Cuturi:2013,Peyre:2015,PeyreCuturi:2019}.

The main technical challenge of the fractional extension \eqref{eqn:JKO-intro} of the classical JKO scheme is to deal with the nonlocality
of the fractional derivative $\partial_t^\alpha \rho$. Numerically, this is overcome by adopting one extremely popular
fractional analogue of the backward Euler scheme (used in the JKO scheme) from the numerical analysis community, known
as the L1 scheme \cite{LinXu:2007} for discretizing the Caputo derivative $\partial_t^\alpha\rho$, and its weights
enter into the term $\overline{\rho}^{n-1}$. Naturally, the nonlocality of the term $\overline{\rho}^{n-1}$ also requires
substantial adaptation of known techniques \cite{JordanKinderlehrerOtto:1998} for the convergence analysis. The gradient
flow formulation and its convergence analysis represent the main contributions of this work.

The rest of the paper is organized as follows. In Section \ref{sec:prelim}, we recall preliminaries
on fractional calculus and describe the connection of the model \eqref{eq:fractionalFP} with stochastic
process and related results on existence and uniqueness. Then in Section \ref{sec:L1}, we describe the L1 scheme,
which is an extension of the backward Euler method to the fractional case, and derive relevant
approximation properties, which are needed for constructing the scheme \eqref{eqn:JKO-intro} and
its convergence analysis. In Section \ref{sec:JKO},
we describe the time-fractional JKO scheme, and state the main theorem, whose lengthy and technical proof is given in
Section \ref{sec:conv}. Last, in Section \ref{sec:numer}, we present numerical results for one- and two-dimensional
problems to illustrate features of the proposed JKO scheme. Below, $C$ denotes a generic constant that depends
on the parameters of the problem, on the initial datum $\rho_0$, and may change at each occurrence, but it is
always independent of the time level $n$ and of time step size $\tau$.

\section{Preliminaries}\label{sec:prelim}

In this section we briefly recall preliminaries on fractional calculus, stochastic model for fractional FPEs
and the concept of weak solution for problem \eqref{eq:fractionalFP}.

\subsection{Preliminaries on fractional calculus}
First, we recall basic concepts from fractional calculus \cite{KilbasSrivastavaTrujillo:2006}.
Throughout, we always assume $\gamma\in [0,1)$, and $a<b$. Then for a function $f: (a,b)
\to\mathbb{R}$, the left-sided and right-sided Riemann-Liouville fractional integrals of
order $\gamma$, denoted by $_aI_t^\gamma f$ and $_tI_b^\gamma f$, are respectively defined by
\begin{align*}
  _aI_t^\gamma f(t)  = \frac{1}{\Gamma(\gamma)}\int_a^t(t-s)^{\gamma-1}f(s)\d s\quad \mbox{and}\quad
  _tI_b^\gamma f(t)  = \frac{1}{\Gamma(\gamma)}\int_t^b(s-t)^{\gamma-1}f(s)\d s.
\end{align*}
These integral operators are well defined for $f\in L^1(a,b)$ and are bounded on $L^p(a,b)$
for any $p\geq1$. The integral operators $_aI_t^\gamma$ and
$_tI_b^\gamma$ are adjoint to each other with respect to $L^2(a,b)$:
\begin{equation}\label{eqn:RL-int-by-part}
  \int_a^b ({_aI_t^\gamma f})(t)g(t)\d t = \int_a^b f(t)({_tI_b^\gamma} g)(t)\d t.
\end{equation}
This relation can be verified directly by changing the order of integration.

The left-sided and right-sided Caputo derivative of order $\alpha\in(0,1)$ of a function $f:(a,b)
\to \mathbb{R}$, denoted by $_aD_t^\alpha f$ and $_tD_b^\alpha f$, are respectively defined by
\begin{align*}
  _aD_t^\alpha f(t)  = ({_aI_t^{1-\alpha} f'})(t)\quad\mbox{and}\quad
  _tD_b^\alpha f(t)  = - ({_tI_b^{1-\alpha} f'})(t).
\end{align*}

Note that the definition of the Caputo derivative of order $\alpha$ requires the existence of a first-order
derivative. Hence, the definition is more stringent. There have been several important efforts in relaxing
the regularity requirement \cite{GorenfloLuchkoYamamoto:2015,LiLiu:2018}. It can be verified that as $\alpha\to 1^-$,
$_aD_t^\alpha f$ recovers the usual first-order derivative $\partial_t f$, when $f$ is sufficiently smooth.
Due to the nonlocality of the fractional derivatives, many useful rules in calculus are no longer
available. The following integration by parts formula is useful.

\begin{lemma}\label{lem:int-by-part}
The following identity holds for $f,g\in C^1[a,b]$ with $g(b)=0$:
\begin{equation}\label{eq: integration by parts}
\int_a^b ({_aD^\alpha_t f})(t) g(t)\,\d t=\int_a^b f(t)({_tD^\alpha_b} g)(t)\,\d t-\frac{f(a)}{\Gamma(1-\alpha)}\int_a^b(t-a)^{-\alpha}g(t)\d t.
\end{equation}
\end{lemma}
\begin{proof}
Indeed, there holds
\begin{align*}
   \int_a^b ({_aD^\alpha_t f})(t) g(t)\,\d t & = \int_a^b ({_aI_t^{1-\alpha} f'})(t)g(t)\d t
      = \int_a^bf'(t)({_tI_b^{1-\alpha}}g)(t)\d t\\
     & = -\int_a^bf(t)({_tI_b^{1-\alpha}}g)'(t)\d t + \big[ f(t)({_tI_b^{1-\alpha}}g)(t)\big]_{a}^b.
\end{align*}
where the first identity follows from the definition of the Caputo derivative $_aD_t^\alpha f$, the second identity follows from \eqref{eqn:RL-int-by-part}, and the third identity is obtained by integration by parts. Since
$g(b)=0$, by the definition of the right-sided Caputo derivative, $-({_tI_b^{1-\alpha}}g)'(t)=
{_tD_b^\alpha}g$ \cite[(2.4.10), p. 91]{KilbasSrivastavaTrujillo:2006}. Then the desired assertion follows by
\begin{align*}
   \big[ f(t)({_tI_b^{1-\alpha}}g)(t)\big]_{a}^b &= f(b)({_tI_b^{1-\alpha}}g)(b) - f(a)({_tI_b^{1-\alpha}}g)(a)\\
    & = -\frac{f(a)}{\Gamma(1-\alpha)}\int_a^b (s-a)^{-\alpha}g(s)\d s,
\end{align*}
since $ ({_tI_b^{1-\alpha}}g)(b)=0$. This completes the proof of the lemma.
\end{proof}

Below we shall write $\partial_t^\alpha f$ and $D_t^\alpha f$ for $
_0D_T^\alpha f$ and $_tD_T^\alpha f$, respectively, for notational simplicity.

\subsection{From stochastic processes to time-fractional FPE}
It is well-known that the classical FPE
\begin{equation}\label{eqn:FPE}
\partial_t f=\div(\nabla\Psi f)+\Delta f,
\end{equation}
which corresponds to problem \eqref{eq:fractionalFP} with $\alpha=1$, is the Kolmogorov forward
equation of the following stochastic differential equation (SDE):
\begin{equation}\label{eq: SDE}
\d X(t)=-\nabla\Psi(X(t))\,\d t+\sqrt{2}\d W(t),\quad \mbox{with }X(0)=X_0,
\end{equation}
where $W(t)$ is a standard $d$-dimensional Wiener process and $X_0$ is a $d$-dimensional random
vector distributed according to the density $\rho_0$. The SDE \eqref{eq: SDE} describes the motion of a
particle undergoing diffusion in an external field $\Psi$, where $X(t)$ is the position of the
particle at time $t$, and the FPE \eqref{eqn:FPE} describes the time evolution of the
PDF of the particle. Its solution $f(t,x)$ is the PDF of finding the
particle at time $t$ and at position $x$. The time-fractional FPE
\eqref{eq:fractionalFP} can be viewed as the Kolmogorov forward equation of a stochastic
process which is obtained from \eqref{eq: SDE} under a time-changed process. Specifically, let
$U_\alpha(t)$ be the $\alpha$-stable subordinator with its Laplace transform given by $\mathbb{E}\big[e^{-kU_\alpha(\tau)}
\big]=e^{-\tau k^\alpha}$, $0<\alpha<1$, and let $S_\alpha(t)$ be the inverse $\alpha$-stable subordinator
\[
S_\alpha(t)=\inf \{\tau>0:~ U_\alpha(\tau)>t\}.
\]
Define the time-changed process
\begin{equation*}
  Y(t)=X(S_\alpha(t)).
\end{equation*}
Then the probability density function (PDF) $p(x,t)$ of $Y(t)$ satisfies
the time-fractional FPE \eqref{eq:fractionalFP}. In fact, the following theorem~\cite{MagdziarzWeronWeron:2007,
HanKobayashiUmarov2011,MagdziarzGajda:2014} describes a close connection between the solutions of
\eqref{eq:fractionalFP} and \eqref{eqn:FPE}.

\begin{theorem} Let $f(x,\tau)$ and $g(\tau,t)$ be respectively the PDFs of $X(\tau)$ and $S(t)$. The following assertions hold.
\begin{enumerate}
\item[$\rm(i)$] The PDF $p(x,t)$ of $Y(t)$ is given by $p(x,t)=\int_0^\infty f(x,\tau)g(\tau, t)\d \tau$.
\item[$\rm(ii)$] The Laplace transform of $p$ and $f$, denoted by $\hat p$ and $\hat f$, respectively, satisfy
$\hat p(x,k)=k^{\alpha-1}\hat f(x,k^\alpha)$.
\item[$\rm(iii)$] $p(x,t)$ is a weak solution to the time-fractional FPE~\eqref{eq:fractionalFP} in the sense of Definition \eqref{def: weak sol} below.
\end{enumerate}
\end{theorem}

See the works \cite{MagdziarzWeronWeron:2007,HanKobayashiUmarov2011,MagdziarzGajda:2014} for further details
on the stochastic representation of problem \eqref{eq:fractionalFP}.

\begin{remark}
There are alternative equivalent reformulations of problem \eqref{eq:fractionalFP}. One popular alternative reads
\begin{equation}
\label{eq: equivalent form}
  \partial_t\rho = {^R\partial_t^{1-\alpha}}(\nabla\cdot(\rho\nabla\Psi) +\Delta \rho),
\end{equation}
where the $^R\partial_t^{1-\alpha}\varphi$ denotes Riemann-Liouville fractional derivative
of order $1-\alpha$, i.e., $^R\partial_t^{1-\alpha}\varphi(t)=\frac{\d}{\d t} ({_0I_t^{\alpha}} \varphi)(t)$.
Formally, it can be obtained from \eqref{eq:fractionalFP} by
applying $^R\partial_t^{1-\alpha}$ to both sides of \eqref{eq:fractionalFP} as
\begin{align*}
  ^R\partial_t^{1-\alpha} {\partial_t^\alpha} \varphi(t)  = \frac{\d}{\d t} {_0I_t^{\alpha}} {_0I_t^{1-\alpha}} \varphi'(t)
    = \frac{\d}{\d t}{_0I_t} \varphi'(t) = \varphi'(t),
\end{align*}
where the first identity is due to the definitions of the fractional derivatives and the second identity is
due to the semigroup property of Riemann-Liouville fractional integral.
Further, one may change the order the spatial and temporal derivative when the forcing $\Psi$
is time-independent. We refer to the work \cite{HenryLanglandsStraka:2010} for discussions on the proper formulation for
a time-dependent forcing. In the present work, we focus on the formulation \eqref{eq:fractionalFP}, and
leave the study of other time-fractional FPE models to future works.

\end{remark}

\subsection{Well-posedness}
Throughout, we only consider probability measures on $\R^d$ that are absolutely continuous with respect to
Lebesgue measure, and often identify a probability measure with its density, as the classical
setting \cite{JordanKinderlehrerOtto:1998}. We denote by $\P_2(\R^d)$
the set of all probability measures on $\R^d$ with a finite second moment, i.e.,
\begin{equation*}
\P_2(\R^d):=\Big\{\rho:\R^d\rightarrow [0,\infty)~\text{measurable},~ \int_{\R^d}\rho(x)\,\d x=1,~ M_2(\rho)<\infty\Big\},
\end{equation*}
where the second moment $M_2(\rho)$ is defined by
\begin{equation}
M_2(\rho)=\int_{\R^d}|x|^2\rho(x)\,\d x.
\end{equation}
Now, we introduce a notion of weak solutions to problem \eqref{eq:fractionalFP}.
Similar to the classical setting, we multiply equation~\eqref{eq:fractionalFP} by a smooth test
function and using the integration by parts formula \eqref{eq: integration by parts} in Lemma
\ref{lem:int-by-part}, which leads to the following notion of weak solution. Below we shall
write a function $f(t,x)$ as a vector valued function $f(t)$.
\begin{definition}
\label{def: weak sol}
A function $\rho\in L^1({\R^+\times\R^d})$ is called a weak solution of problem
\eqref{eq:fractionalFP} with initial datum $\rho_0\in \P_2(\R^{d})$ if it satisfies
that for any $\varphi\in C^\infty([0,T]\times
\mathbb{R}^d)$ with $\varphi(T)=0$, there holds
\begin{align}
\label{eq: weak formulation}
\int_0^T\int_{\R^d}\Big({_tD_T^\alpha}\varphi(t)+\nabla\Psi\cdot&\nabla\varphi(t)-\Delta\varphi(t)\Big)\rho(t)\d x\,\d t\\
   &=\frac{1}{\Gamma(1-\alpha)}\int_{\R^d}\int_0^T t^{-\alpha}\varphi(t)\d t\rho_0\,\d x.\nonumber
\end{align}
\end{definition}

Note that the formulation \eqref{eq: weak formulation} of the weak solution involves a
nonlocal term $\int_{\R^d}\leftidx{_t} I^{1-\alpha}_T\varphi(0)\rho_0\,\d x$. This term appears due
to the nonlocality of the Caputo derivative $\partial_t^\alpha\rho$, cf. Lemma \ref{lem:int-by-part}. In
the limit $\alpha\to 1^-$, it recovers the usual $\int_{\R^d}\varphi(0)\rho_0\,\d x$, in view of the
identity $\lim_{\alpha\to0^+}{_0I_t^{\alpha-1}}\varphi(0)=\varphi(0)$, under suitable regularity assumptions.
We are not are aware of any existing work directly investigating the existence and regularity of
the solutions on problem \eqref{eq:fractionalFP}. However, the existence and uniqueness of the weak solution
of an equivalent formulation given in \eqref{eq: equivalent form}  of problem \eqref{eq:fractionalFP}
were already proven in \cite[Theorem 3.3]{CamilliDeMaio:2017}. See also \cite{Zacher:2009,Akagi:2019} for
discussion on the well-posedness (existence and uniqueness) of abstract Volterra type evolution equations in
a Hilbert space setting. It is also worth noting that the proper interpretation of the initial condition
requires some care; see the works \cite{GorenfloLuchkoYamamoto:2015,LiLiu:2018} for in-depth discussions.
We leave a detailed study on these important analytic issues (possibly in more general settings of metric spaces and
spaces of probability measures) to future works.

\section{Numerical approximation of Caputo derivative}\label{sec:L1}

Now we recall the numerical approximation of the Caputo derivative $\partial_t^\alpha \varphi(t)$. There are
several different ways to construct a ``fractional'' analogue of the classical backward Euler method (see \cite{JinLazarovZhou:2019}
for a concise overview), on which the classical JKO scheme \cite{JordanKinderlehrerOtto:1998} is based.
We shall employ the so-called piecewise linear approximation, commonly known as the L1 approximation (due to
Lin and Xu \cite{LinXu:2007}) in the numerical analysis literature.

Consider a uniform partition of the time interval $[0,T]$, with a time step size $\tau= \frac{T}{N}$ and
the grid $t_n=n \tau$, $n=0,1,\ldots,N$. For any function $\varphi\in C[0,T]$, we use the shorthand notation
$\varphi^n=\varphi(t_n)$. Further, we denote $C_\alpha =\Gamma(2-\alpha)^{-1}$.
Then the L1 approximation \cite{LinXu:2007} is constructed as follows. First we split the interval $[0,t_{n}]$ into $n$
subintervals
\begin{equation*}
  \partial^\alpha_t\varphi^{n}=\frac{1}{\Gamma(1-\alpha)}\sum_{i=1}^{n}\int_{t_{i-1}}^{t_i}(t_{n}-s)^{-\alpha}\varphi'(s)\,\d s,
\end{equation*}
and then by approximating $\varphi$ by its linear interpolation over the subinterval $[t_{i-1},t_{i}]$, i.e.,
\begin{equation*}
   \varphi(t)\approx\frac{t_{i}-t}{\tau}\varphi^{i-1}+\frac{t-t_{i-1}}{\tau}\varphi^{i},\quad t\in[t_{i-1},t_{i}], i=1,\ldots, N,
\end{equation*}
or equivalently $\varphi'(t)\approx (\varphi^i-\varphi^{i-1})/\tau$ for $t\in [t_{i-1},t_i]$, we obtain the following
approximation to the Caputo derivative $\partial_t^\alpha \varphi$ at time $t=t_{n}$ by
\begin{align*}
\partial^\alpha_t\varphi^{n}=\frac{1}{\Gamma(1-\alpha)}\sum_{i=1}^n\int_{t_{i-1}}^{t_{i}}(t_{n}-s)^{-\alpha}\frac{\varphi^{i}-\varphi^{i-1}}{\tau}\,\d s + r_\tau^{n},
\end{align*}
where $ r_\tau^{n}$ is the local truncation error. It can be verified that $r_\tau^n$ takes the following
form \cite{LinXu:2007}
\begin{equation*}
  r_\tau^{n}\leq c_\varphi \left[\frac{1}{\Gamma(1-\alpha)}\sum_{i=1}^n\int_{t_{i-1}}^{t_i}\frac{t_{i}+t_{i-1}-2s}{(t_{n}-s)^\alpha}\d s + O(\tau^2)\right],
\end{equation*}
with the constant $c_\varphi$ depending only on $\|\varphi\|_{C^2[0,T]}$.
Now the elementary integral
\begin{equation*}
  \int_{t_{i-1}}^{t_{i}} (t_{n}-s)^{-\alpha} \d s = (1-\alpha)^{-1}\tau^{1-\alpha}((n+1-i)^{1-\alpha}-(n-i)^{1-\alpha})
\end{equation*}
and simple algebraic manipulations (with $C_\alpha=1/\Gamma(2-\alpha)$) lead to
\begin{align}
  \partial^\alpha_t\varphi^{n}& \approx C_\alpha \tau^{-\alpha} \sum_{i=1}^n(\varphi^{i}-\varphi^{i-1})((n+1-i)^{1-\alpha}-(n-i)^{1-\alpha})\nonumber\\
    &  = C_\alpha\tau^{-\alpha}\sum_{i=0}^{n} b_{n-i}^{(n)}\varphi^i:=\bar\partial_\tau \varphi^{n},\label{eqn:L1}
\end{align}
where the quadrature weights $b_i^{(n)}$ are given by
\begin{equation}\label{eqn:bi}
  b_i^{(n)} = \left\{\begin{array}{ll}
    1, & i=0,\\
    (i+1)^{1-\alpha}+(i-1)^{1-\alpha}-2i^{1-\alpha},   & i=1,\ldots,n-1,\\
    (n-1)^{1-\alpha}-n^{1-\alpha},   & i = n.
  \end{array}\right.
\end{equation}
Note that the last weight $b_{n}^{(n)}$ depends on $n$ differently than the preceding ones. In the special case
$\alpha=1$, the approximation reduces to the classical backward Euler method, since $b_0^{(n)}=1$ and $b_1^{(n)}=-1$,
and $b_i^{(n)}=0,$ for any $1<i\leq n$. In a similar manner, the
L1 approximation $\overline{D}_\tau^\alpha \varphi^n$ to the right-sided Caputo fractional derivative
$_tD_T^\alpha \varphi(t)$ at $t=t_n$ is given by
\begin{equation}\label{eqn:L1-right}
  \begin{aligned}
    \overline{D}_\tau^\alpha \varphi^n & = C_\alpha\tau^{-\alpha} \sum_{j=0}^{N-n} b_{j}^{(N-n)}\varphi^{n+j}
        = C_\alpha \tau^{-\alpha}\sum_{j=n}^Nb_{j-n}^{(N-n)}\varphi^j.
  \end{aligned}
\end{equation}
This approximation can be obtained by a simple change of variables.

By construction, the L1 approximations $\bar\partial_\tau^\alpha \varphi^n$ and $\overline{D}_\tau^\alpha\varphi^n$
are essentially a weighted piecewise linear approximation, with respect to the weakly singular weight $t^{-\alpha}$.
The discrete approximations are of convolution form, similar to the continuous fractional derivatives $\partial_t^\alpha
\varphi$ and $D_t^\alpha\varphi$. The L1 approximation has been widely employed for solving time-fractional diffusion,
due to its excellent empirical performance; see \cite{JinLazarovZhou:2016ima,JinZhou:2017,JinLiZhou:2018sinum} for some relevant works on
error analysis.

We will need the following auxiliary lemma.
\begin{lemma}
\label{lem: properties of b}
For $0<\alpha<1$ and a fixed $n\in\mathbb{N}$, for the weights  $b_j^{(n)}$ given in \eqref{eqn:bi},
then there holds $b_i^{(n)}< 0$ for $i=1,\ldots, n$ and~~$\sum_{i=0}^nb_i^{(n)}=0$. Further,
\begin{align*}
  \sum_{n=1}^k(-b_{n}^{(n)}) & = k^{1-\alpha}\quad \mbox{and}\quad
  \sum_{j=1}^{k-i}(-b_{j}^{(j+i)}) = 1 + (k-i)^{1-\alpha}-(k-i+1)^{1-\alpha}.
\end{align*}
\end{lemma}
\begin{proof}
The first assertion is well known (see, e.g., \cite[eq. (3.7)]{LinXu:2007}), and we only give a proof
for completeness. Consider the function $f(x)=x^{1-\alpha}$ for $x>0$. Since $0<\alpha<1$, we have
$f''(x)=-\alpha(1-\alpha)x^{-\alpha-1}<0$, and hence $f$ is strictly concave on $(0,\infty)$. By Jensen's inequality we have
\begin{align*}
i^{1-\alpha}&=f(i)=f\big(\tfrac{i+1 + i-1}{2}\big)\\
&> \tfrac12f(i+1)+\tfrac12f(i-1)\\
 &=\tfrac{1}{2}(i+1)^{1-\alpha}+\tfrac12(i-1)^{1-\alpha},
\end{align*}
which immediately implies that $b_i^{(n)}< 0$ for all $i=1,\ldots, n$. Further, straight computations give
\begin{align*}
\sum_{i=0}^{n}b_i^{(n)}=1+\sum_{i=1}^{n-1} \big((i+1)^{1-\alpha}+(i-1)^{1-\alpha}-2i^{1-\alpha}\big) + ((n-1)^{1-\alpha}-n^{1-\alpha})=0.
\end{align*}
This shows the second assertion. The rest follows from straightforward computation as:
\begin{align*}
  \sum_{n=1}^k(-b_{n}^{(n)}) & = \sum_{n=1}^k (n^{1-\alpha}-(n-1)^{1-\alpha})= k^{1-\alpha},\\
  \sum_{j=1}^{k-i}(-b_{j}^{(j+i)}) & =  -\sum_{j=1}^{k-i}\big((j+1)^{1-\alpha}+(j-1)^{1-\alpha}-2j^{1-\alpha}\big)\\
                                & = -\sum_{j=1}^{k-i}\big(((j+1)^{1-\alpha}-j^{1-\alpha})-(j^{1-\alpha}-(j-1)^{1-\alpha})\big)\\
                                & = 1 + (k-i)^{1-\alpha}-(k-i+1)^{1-\alpha}.
\end{align*}
This completes the proof of the lemma.
\end{proof}

We will also need the following useful inequality of Gronwall type \cite[Lemma 2.2]{LiaoLiZhang:2018}.
\begin{lemma}\label{lem:Gronwall}
Suppose $\{\phi^n\}_{n=0}^N$ are nonnegative, and satisfy the following inequality (with $\tau=T/N$)
\begin{equation*}
  \bar\partial_\tau^\alpha \phi^n \leq C_1 + C_2\phi^n,
\end{equation*}
where $C_1,C_2$ are positive constants. Then there holds
\begin{equation*}
  \phi^n \leq 2E_{\alpha}(2C_2t_n^\alpha)\Big(\phi^0+\frac{C_1}{\Gamma(1+\alpha)}t_n^\alpha\Big),\quad \forall n=1,\ldots, N
\end{equation*}
where $E_\alpha$ denotes the Mittag-Leffler function $E_\alpha(z)=\sum_{k=0}^\infty\frac{z^k}{\Gamma(k\alpha+1)}$.
\end{lemma}

The next result gives a ``semi-discrete'' version of the integration by parts formula in
Lemma \ref{lem:int-by-part} for the L1 approximation $\bar\partial_\tau^\alpha\varphi^n$.
\begin{lemma}\label{lem:int-by-part-semi}
Let $\{\varphi^n\}_{n=0}^N$ be a given sequence, and $\phi(t)\in C^1[0,T]$ with $\phi(T)=0$. The piecewise constant approximation
$\varphi_\tau(t)$ is defined by $\varphi_\tau(t)=\phi^n$ for $(n-1)\tau <t\leq n\tau$,
with $\phi_\tau(0)=\phi^0$. Then the following identity holds
\begin{align*}
    \int_{0}^{T}(\bar \partial_\tau^\alpha \varphi^n)(t) \phi(t)\d t
  =& \int_0^T \varphi_\tau(t) {\overline{D}_\tau^\alpha \phi(t)}\d t  + C_\alpha\tau^{-\alpha}\varphi^0\sum_{n=1}^{N}b_n^{(n)}\int_{t_{n-1}}^{t_n} \phi(t)\d t,
\end{align*}
where the function ${\overline{D}_\tau^\alpha}\phi(t)$ is defined by {\rm(}with zero extension on $\phi${\rm)}
\begin{equation*}
  {\overline{D}_\tau^\alpha}\phi(t) = \sum_{i=n}^{N}b_{i-n}^{(N-n)}\phi(t+(i-n)\tau), \quad \forall t\in (t_{n-1},t_n], \ \ n=1,\ldots, N.
\end{equation*}
\end{lemma}
\begin{proof}
By the definition of the L1 approximation in \eqref{eqn:L1}, we have
\begin{align*}
C_\alpha^{-1}\tau^\alpha {\rm LHS}&=\sum_{n=1}^N\int_{t_{n-1}}^{t_n}\Big[\varphi^n+\sum_{i=0}^{n-1}b_{n-i}^{(n)}\varphi^i\Big]\phi(t)\d t\\
&=\sum_{n=1}^{N}\int_{t_{n-1}}^{t_n}\varphi^n\phi(t)\d t + \sum_{n=1}^N\sum_{i=1}^{n-1}b_{n-i}^{(n)}\int_{t_{n-1}}^{t_n}\varphi^i\phi(t)\d t
+\sum_{n=1}^Nb_n^{(n)}\int_{t_{n-1}}^{t_n}\varphi^0 \phi(t)\d t\\
&=:  {\rm I} + {\rm II} + {\rm III}.
\end{align*}
By the definition of the interpolation $\varphi_\tau(t)$, the first term $\rm I$ can be rewritten as
\begin{equation*}
{\rm I}=\sum_{n=1}^N\int_{t_{n-1}}^{t_n}\varphi_\tau(t)\phi(t)\d t.
\end{equation*}
Now we turn to the term ${\rm II}$. Using the change of variables $t\mapsto t+(n-i)\tau$
and then applying the definition of the interpolation $\varphi_\tau(t)$, we deduce
\begin{align*}
{\rm II}&=\sum_{n=1}^N\sum_{i=1}^{n-1}b_{n-i}^{(n)}\int_{t_{i-1}}^{t_i}\varphi^i\phi(t+(n-i)\tau)\d t\\
&=\sum_{n=1}^N\sum_{i=1}^{n-1}b_{n-i}^{(n)}\int_{t_{i-1}}^{t_i}\varphi_\tau(t)\phi(t+(n-i)\tau)\d t.
\end{align*}
Next we interchange the order of summation and relabel the indices (with the convention that the sum is zero
when the lower index is greater than the upper index) to obtain
\begin{align*}
{\rm II}=&\sum_{i=1}^{N-1}\sum_{n=i+1}^Nb_{n-i}^{(n)}\int_{t_{i-1}}^{t_i}\varphi_\tau(t)\phi(t+(n-i)\tau)\d t\\
=&\sum_{n=1}^{N-1}\sum_{i=n+1}^Nb_{i-n}^{(i)}\int_{t_{n-1}}^{t_n}\varphi_\tau(t)\phi(t+(i-n)\tau)\d t\\
=&\sum_{n=1}^N\sum_{i=n+1}^Nb_{i-n}^{(i)}\int_{t_{n-1}}^{t_n}\varphi_\tau(t)\phi(t+(i-n)\tau)\d t.
\end{align*}
Now recall the definition of the weights $b_{i-n}^{(i)}$ in \eqref{eqn:bi}, there holds
\begin{equation*}
  b_{i-n}^{(i)} = b_{i-n}^{(N-n)},\quad i = n+1,\ldots, N-1.
\end{equation*}
Further, since $\phi$ is supported on $(0,T)$, we may change $b_{N-n}^{(N)}$ to $b_{N-n}^{(N-n)}$, and thus obtain
\begin{equation*}
  {\rm II}=\sum_{n=1}^N\sum_{i=n+1}^Nb_{i-n}^{(N-n)}\int_{t_{n-1}}^{t_n}\varphi_\tau(t)\phi(t+(i-n)\tau)\d t.
\end{equation*}
Consequently, since $b_{0}^{N-n}=1$ and using the definition of the notation $\overline{D}_\tau^\alpha\phi(t)$,
\begin{align*}
  {\rm I} + {\rm II} &= \sum_{n=1}^N\int_{t_{n-1}}^{t_n}\varphi_\tau\Big(\phi(t)+\sum_{i=n+1}^{N}b_{i-n}^{(N-n)}\phi(t+(i-n)\tau)\Big)\d t\\
                     &= \sum_{n=1}^N\int_{t_{n-1}}^{t_n}\varphi_\tau {\overline{D}_\tau^\alpha \phi(t)}\d t = \int_{0}^{T}\varphi_\tau {\overline{D}_\tau^\alpha \phi(t)}\d t.
\end{align*}
Then combining the preceding identities completes the proof of the lemma.
\end{proof}

The following result gives the error estimates of the L1 approximation for smooth functions.
\begin{theorem}\label{thm:err-approx}
The following error estimates hold
\begin{align*}
  \bar\partial_\tau^\alpha \varphi^{n}& = (\partial_t^\alpha \varphi)(t_n)  + O(\tau^{2-\alpha}),\quad \forall \varphi\in C^2[0,T]\\
  C_\alpha\tau^{-\alpha} \sum_{n=1}^{N}b_n^{(n)} \int_{t_{n-1}}^{t_n}
  \varphi(t)\d t &= -({_tI_T^\alpha \varphi})(0) + O(\tau),\quad \forall \varphi\in C^1[0,T].
\end{align*}
\end{theorem}
\begin{proof}
The first estimate can be found at \cite[equations (3.12) and (3.13)]{LinXu:2007}.
It suffices to show the second estimate. Using the expression of the weight $b_n^{(n)}$,
we may rewrite the left hand side as (with $C_\alpha'=-C_\alpha(1-\alpha)=-\frac{1}{\Gamma(1-\alpha)}$
\begin{align*}
  {\rm LHS} &= C_\alpha'\sum_{n=1}^{N}\int_{t_{n-1}}^{t_n}t^{-\alpha}\d t \Big(\tau^{-1}\int_{t_{n-1}}^{t_n}
  \varphi(s)\d s\Big)\\
   &=C_\alpha'\sum_{n=1}^{N}\int_{t_{n-1}}^{t_n}t^{-\alpha}\varphi(t)\d t - C_\alpha'\sum_{n=1}^{N}\int_{t_{n-1}}^{t_n}t^{-\alpha}\Big( \varphi(t)-\tau^{-1}\int_{t_{n-1}}^{t_n}\varphi(s)\d s\Big)\d t\\
   &=C_\alpha'\int_0^{T}t^{-\alpha}\varphi(t)\d t - C_\alpha'\sum_{n=1}^{N}\int_{t_{n-1}}^{t_n}t^{-\alpha}\Big( \varphi(t)-\tau^{-1}\int_{t_{n-1}}^{t_n}\varphi(s)\d s\Big)\d t.
\end{align*}
Next we bound the term in the bracket by
\begin{align*}
   |\varphi(t)-\tau^{-1}\int_{t_{n-1}}^{t_n}\varphi(s)\d s| & = |\tau^{-1}\int_{t_{n-1}}^{t_n}\varphi(t)-\varphi(s)\d s |
    \leq \|\varphi\|_{C^1[0,T]}\tau.
\end{align*}
Combining the last two estimates gives the desired estimate.
\end{proof}

\section{Time-fractional JKO  scheme}\label{sec:JKO}

Now we construct a JKO type scheme for problem \eqref{eq:fractionalFP},
and give the main result of the work.

\subsection{Wasserstein distance}
The Wasserstein distance of order two, denoted by $W_2(\mu_1,\mu_2)$, between  two (Borel) probability measures
$\mu_1$ and $\mu_2$ on $\mathbb{R}^d$ is defined by
\begin{equation}\label{eqn:Wasserstein}
  W_2(\mu_1,\mu_2) ^2  = \inf_{p\in \mathcal{P}(\mu_1,\mu_2)}\int_{\mathbb{R}^d\times\mathbb{R}^d}|x-y|^2p(\d x\d y),
\end{equation}
where $\mathcal{P}(\mu_1,\mu_2)$ is the set of all probability measures on $\mathbb{R}^{d}\times\mathbb{R}^d$
with the first marginal $\mu_1$ and second marginal $\mu_2$, and the symbol $|\cdot|$ denotes the usual
Euclidean norm on $\mathbb{R}^d$. That is, a probability measure $p$ is in $\mathcal{P}(\mu_1,\mu_2)$ if
and only if for each Borel subset $A\subset\mathbb{R}^d$ there holds
\begin{equation*}
  p(A\times\mathbb{R}^d)=\mu_1(A)\quad \mbox{and}\quad p(\mathbb{R}^d\times A) = \mu_2(A).
\end{equation*}
It is well known that $W_2(\mu_1,\mu_2)$ defines a metric on the set of probability measure $\mu$ on
$\mathbb{R}^d$ having finite second moments: $\int_{\mathbb{R}^d}|x|^2\mu(\d x)<\infty$ \cite{GivensShortt:1984}.

The variational problem \eqref{eqn:Wasserstein} is an example of a Monge-Kantorovich mass transport
problem with a cost function  $c(x,y)=|x-y|^2$. In that context, an infimizer $p^*$ is referred to
as an optimal (transport) plan; see \cite{GivensShortt:1984} for a probabilistic proof that the
infimum in \eqref{eqn:Wasserstein} exists when the measures $\mu_1$ and $\mu_2$ have finite second moments. Brenier
\cite{Brenier:1991} established the existence of a one-to-one optimal (transport) plan in the case that the measures
$\mu_1$ and $\mu_2$ have bounded support and are absolutely continuous with respect to Lebesgue measure.

\subsection{Time-fractional JKO scheme}

Next we derive the fractional analogue of the JKO scheme for problem \eqref{eq:fractionalFP}. The classical JKO scheme
\cite{JordanKinderlehrerOtto:1998} for the FPE \eqref{eqn:FPE} is based on the backward Euler approximation
of the first-order derivative $\partial_t \rho$ in time. Hence, naturally, the fractional analogue relies on a
backward Euler type approximation to the Caputo derivative $\partial_t^\alpha \rho$. We shall employ the L1
approximation \cite{LinXu:2007} described in Section \ref{sec:L1}. By combining the classical JKO scheme
\cite{JordanKinderlehrerOtto:1998} and the L1 approximation of the fractional time derivative $\partial_t^\alpha
\rho$, we obtain a JKO type scheme for problem \eqref{eq:fractionalFP} as follows.
\begin{scheme}[Discrete variational approximation scheme for the model \eqref{eq:fractionalFP}]
\label{sc: app scheme}
Let $\rho^0:=\rho_0$. Given $\rho^0$, find $\rho^{n}$, $n=1,2,\ldots,N$, as the unique minimizer of
\begin{equation}\label{eq: minimization prob}
\frac{C_\alpha}{2\tau^\alpha}W_2^2(\rho,\overline{\rho}^{n-1})+\F(\rho),
\end{equation}
over $\rho\in \P_2(\R^d)$, where $\overline{\rho}^{n-1}$ and $ \F(\rho)$ are defined respectively by
\[
\overline{\rho}^{n-1}:=\sum_{i=0}^{n-1} (-b_{n-i}^{(n)})\rho^i\quad\mbox{and}\quad \F(\rho):=\E(\rho)+\S(\rho),
\]
with
\begin{equation*}
  \E(\rho)=\int_{\R^d} \Psi\rho\d x\quad \mbox{and}\quad \S(\rho) = \int_{\R^d} \rho\log\rho\,\d x.
\end{equation*}
\end{scheme}

Given $\rho^0\in \mathcal{P}_2(\mathbb{R}^d)$, the existence and uniqueness of a minimizer of  Scheme \ref{sc: app scheme}
was proven in \cite[Proposition 4.1]{JordanKinderlehrerOtto:1998}.  In view of Lemma \ref{lem: properties of b}(i),
$\sum_{i=0}^{n-1}(-b_{n-i}^{(n)})=1$, and thus $\overline{\rho}^{n-1}$ is a convex combination of all past approximations
$\{\rho^i\}_{i=0}^{n-1}$. This property plays a crucial role in the convergence analysis. One distinct feature of the
scheme is that instead of using only the immediate previous density
$\rho^{n-1}$ in \eqref{eq: minimization prob} as in the classical JKO-scheme, it employs a \textit{convex} combination
$\overline{\rho}^{n-1}$ of all previous densities $\{\rho^i\}_{i=0}^{n-1}$. This is to capture the memory effect (and
thus non-Markovian nature) of the continuous time-fractional FPE. In the limiting case $\alpha=1$, it is identical
with the classical JKO scheme (see the properties of the L1 approximation in Section \ref{sec:L1}).

Below we shall make one minor assumption on $\Psi$. Note that the assumption
$\Psi(x)\geq 0$ can be relaxed to that $\Psi$ is bounded from below.
\begin{assumption}\label{ass:Psi}
$\Psi(x)\in C^\infty(\mathbb{R}^d)$, $\Psi(x)\geq 0$ and $|\nabla \Psi(x)|\leq C(|x|+1)$ for all $x\in\mathbb{R}^d$.
\end{assumption}

\begin{remark}
There are other possible formulations of JKO type schemes for the time-fractional FPE \eqref{eq:fractionalFP}. For example, the following formulation seems
also feasible. Given $\rho^0$, find $\rho^n$, $n=1,2,\ldots,N$, by minimizing over $\mathcal{P}_2(\mathbb{R}^d)$ the following functional
\begin{equation*}
\frac{C_\alpha}{2\tau^\alpha}\sum_{i=0}^{n-1}(-b_{n-i}^{(n)})W_2^2(\rho,\rho^i)+\F(\rho).
\end{equation*}
By the convexity of the Wasserstein distance, this functional is an upper bound of the one in \eqref{eq: minimization prob}.
However, it involves multiple Wasserstein distances and thus is computationally far less convenient. Thus it is not explored in
this work.
\end{remark}

The next result represents the main theoretical contribution of the paper, i.e., the convergence of the
discrete approximations $\{\rho^n\}_{n=1}^N$. The proof of the theorem is lengthy and technical
and  will be given in Section \ref{sec:conv}.
\begin{theorem}\label{thm: main theorem}
Let $\rho_0\in \P_2(\R^d)$ satisfy $\F(\rho_0)<\infty$. For any fixed $\tau>0$, let $\{\rho^n\}_{n=1}^N$
be the sequence of minimizers given by Scheme~\ref{sc: app scheme}. For any $t\geq 0$, we define a
picewise-constant time interpolation: with $\rho_\tau(0)=\rho^0$ and
\begin{equation}\label{eq: interpolation}
\rho_\tau(t,x)=\rho^n(x)\qquad\text{for}~~(n-1)\tau<t\leq n\tau,\quad n=1,\ldots, N.
\end{equation}
Then under Assumption \ref{ass:Psi}, for any $T>0$,
\begin{equation}\label{eq: weaklimit}
\rho_\tau \to \rho\quad\text{weakly in}\quad L^1\big((0,T)\times\R^d\big)~~~\text{as}~~\tau\to 0,
\end{equation}
where $\rho$ is the unique weak solution to problem \eqref{eq:fractionalFP} in the
sense of Definition~\ref{def: weak sol}.
\end{theorem}

\section{Proof of Theorem \ref{thm: main theorem}}\label{sec:conv}

This section is devoted to the convergence analysis of Scheme \ref{sc: app scheme}, i.e., the proof of Theorem \ref{thm: main theorem}.
First, we give the Euler-Lagrange equation for the sequence of minimizers.
\begin{lemma}\label{lem: EL eqn}
Let $\tau>0$ and $\{\rho^{n}\}_{n=1}^N$ be the sequence of minimizer given by Scheme~\ref{sc: app scheme}, and $P^{n}$
the optimal plan for the Wasserstein distance $W_2(\rho^{n},\overline{\rho}^{n-1})$ between
$\overline{\rho}^{n-1}$ and $\rho^{n}$. Then for all $\varphi\in C_0^\infty(\R^d)$, there holds
\begin{equation}\label{eq: EL eqn}
\frac{C_\alpha}{\tau^\alpha}\int_{\R^{2d}}(y-x)\cdot\nabla\varphi(y)P^n(\d x\d y)+\int_{\R^d}\Big(\nabla\Psi\cdot\nabla\varphi-\Delta\varphi\Big)\rho^{n}(x)\,\d x=0.
\end{equation}
\end{lemma}
\begin{proof}
The derivation of the Euler-Lagrange equation for the sequence $\{\rho^{n}\}_{n=1}^N$ of minimizers  follows the
now well-established procedure; see e.g., \cite{JordanKinderlehrerOtto:1998,Huang:2000,DuongPeletierZimmer:2014}. We only sketch
the main steps below. Let $\xi\in C_c^\infty(\R^d,\R^d)$, and we define a flow $\Phi:[0,\infty)\times
\R^d\to\R^d$ by
\begin{equation*}
  \frac{\partial\Phi_s}{\partial s}=\xi(\Phi_s),\quad\mbox{with }\Phi_0={\rm Id}.
\end{equation*}
For any $s\in\mathbb{R}$, let $\rho_s(y)\d y$ be the push-forward of the measure $\rho^{n}(y)\d y$ under $\Phi_s$. That is, for
any $\zeta\in C_0^\infty(\R^d)$, we have
\begin{equation*}
\int_{\R^d}\rho_s(y)\zeta(y)\,\d y=\int_{\R^d}\rho^{n}(y)\zeta(\Phi_s(y))\,\d y.
\end{equation*}
Since $\Phi_0={\rm Id}$, it follows that $\rho_0(y)=\rho^{n}(y)$ and an explicit calculation yields
\begin{equation*}
\partial_s\rho_s\big\vert_{s=0}=-\div (\rho^{n}\xi).
\end{equation*}
Following the computations in \cite{JordanKinderlehrerOtto:1998},
we derive the following stationarity condition on $\rho^{n}$:
\begin{equation}\label{eq: EL for rhonh}
\frac{C_\alpha}{\tau^\alpha}\int_{\R^{2d}}(y-x)\cdot\xi(y)P^n(\d x\d y)+\int_{\R^d}\Big(\nabla\Psi\cdot\xi-\div\xi\Big)\rho^{n}(x)\d x=0,
\end{equation}
where $P^n$ is the optimal plan in the definition of the Wasserstein distance $W_2(\rho^{n},\overline{\rho}^{n-1})$
between $\bar{\rho}^{n-1}$ and $\rho^{n}$. For any $\varphi\in C_0^\infty(\R^d)$, by choosing $\xi=\nabla\varphi$ in~\eqref{eq: EL for rhonh}, we get
\begin{equation*}
\frac{C_\alpha}{\tau^\alpha}\int_{\mathbb{R}^{2d}}(y-x)\cdot\nabla\varphi(y)P^n(\d x\d y)+\int_{\mathbb{R}^d}\Big(\nabla\Psi\cdot\nabla\varphi-\Delta\varphi\Big)\rho^{n}(x)\d x=0.
\end{equation*}
This completes the proof of the lemma.
\end{proof}

The next result is an immediate corollary of Lemma \ref{lem: EL eqn}.
\begin{corollary}\label{cor:EL}
The solutions $\{\rho^n\}_{n=1}^N$ given by Scheme \ref{sc: app scheme} satisfy for all $\varphi\in C_c^\infty(\mathbb{R}^d)$ and all $n=1,\ldots,N$:
\begin{equation*}
  \left|\int_{\mathbb{R}^d}\frac{C_\alpha(\rho^n-\overline{\rho}^{n-1})}{\tau^\alpha}\varphi\d x + \int_{\mathbb{R}^d} (\nabla \Psi\cdot\nabla\varphi - \Delta \varphi)\rho^n\d x\right|
  \leq \sup_x\frac{\|\nabla^2 \varphi(x)\|}{2}\tau^{-\alpha}W_2^2(\rho^n,\overline{\rho}^{n-1}),
\end{equation*}
where $\nabla^2 \varphi\in\mathbb{R}^{d\times d}$ denotes the Hessian of $\varphi$, and $\|\cdot\|$ denotes the spectral norm of a matrix.
\end{corollary}
\begin{proof}
The assertion follows identically with \eqref{timederivativeappro} and \eqref{epsilon} below, and hence it is omitted.
\end{proof}

In the next few lemmas, we derive several important \textit{a priori} estimates on the sequence $\{\rho^n\}_{n=1}^N$
of approximations. These estimates are analogous to~(42)-(45) in~\cite{JordanKinderlehrerOtto:1998}. However, due to
the appearance of the convex combination density $\overline{\rho}^{n-1}$ instead of $\rho^{n-1}$ in Scheme~\ref{sc: app scheme},
the derivation of these estimates is more involved than that in~\cite{JordanKinderlehrerOtto:1998}.

First, we derive elementary inequalities for $\F(\overline{\rho}^n)$, using convexity of $\F(\rho)$.
\begin{lemma}\label{lem:convex}
For any $n$, there holds
\begin{align}
  \F(\overline{\rho}^{n-1}) & \leq \sum_{i=0}^{n-1} (-b_{n-i}^{(n)})\F(\rho^i),\label{eqn:bdd-F-brho}\\
\sum_{i=1}^{n}\F(\bar{\rho}^{i-1}) &\leq n^{1-\alpha}\F(\rho^0)+\sum_{i=1}^{n-1}\big(1+(n-i)^{1-\alpha}-(n-i+1)^{1-\alpha}\big)\F(\rho^i).\label{eq: bound sum F}
\end{align}
\end{lemma}
\begin{proof}
Since $\F(\overline{\rho}^{n-1})=\E(\overline{\rho}^{n-1})+\S(\overline{\rho}^{n-1})$, for the energy term $\E(\overline{\rho}^{n-1})$, we have
\begin{align*}
\E(\overline{\rho}^{n-1})& =\int_{\mathbb{R}^d} \Psi\overline{\rho}^{n-1}\d x=\int_{\mathbb{R}^d}\Psi\sum_{i=0}^{n-1}(-b_{n-i}^{(n)})\rho^i \d x \\
 &=\sum_{i=0}^{n-1}(- b_{n-i}^{(n)})\int_{\mathbb{R}^d}\Psi\rho^i \d x=\sum_{i=0}^{n-1}(-b_{n-i}^{(n)})\E(\rho^i).
\end{align*}
For the entropy term $\S(\overline{\rho}^{n-1})$: since the function $z\mapsto s(z)= z\log(z)$ is convex for $z\geq 0$
and the identity $-\sum_{i=0}^{n-1}b_{n-i}^{(n)}=1$ (cf. Lemma \ref{lem: properties of b}(i)), Jensen's inequality implies
\[
s(\overline{\rho}^{n-1})=s\Big(-\sum_{i=0}^{n-1}b_{n-i}^{(n)}\rho^i\Big)\leq \sum_{i=0}^{n-1}(-b_{n-i}^{(n)}) s(\rho^i),
\]
which, upon integration, immediately implies
\begin{equation*}
\S(\overline{\rho}^{n-1})=\int_{\mathbb{R}^d} s(\overline{\rho}^{n-1})\d x\leq \sum_{i=0}^{n-1}(-b_{n-i}^{(n)})\int_{\mathbb{R}^d} s(\rho^i)\d x =\sum_{i=0}^{n-1} (-b_{n-i}^{(n)})\S(\rho^i).
\end{equation*}
Then the preceding two estimates imply
\begin{align*}
\F(\overline{\rho}^{n-1}) & =\E(\overline{\rho}^{n-1}) + \S(\overline{\rho}^{n-1}) \\
   & \leq \sum_{i=0}^{n-1} (-b_{n-i}^{(n)})(\E(\rho^i) + \S(\rho^i))=\sum_{i=0}^{n-1} (-b_{n-i}^{(n)})\F(\rho^i).
\end{align*}
This shows the first assertion. Next, summing the inequality over $i=1$ to $i=n\leq N$, changing the
order of summation and relabeling the indices yield
\begin{align*} 
\sum_{i=1}^{n}\F(\bar{\rho}^{i-1}) & \leq \sum_{i=1}^{n}\sum_{j=0}^{i-1}(-b_{j-i}^{(j)})\F(\rho^i)= \sum_{i=0}^{n-1} \sum_{j=i+1}^n(-b_{j-i}^{(j)})\F(\rho^i)\nonumber\\
 & =\Big(\sum_{i=1}^n(-b_i^{(i)})\Big)\F(\rho^0)+\sum_{i=1}^{n-1}\Big(\sum_{j=1}^{n-i}(-b_{j}^{(j+i)})\Big) \F(\rho^i).
\end{align*}
Upon noting the identities in Lemma \ref{lem: properties of b}, we obtain the second assertion.
\end{proof}

The next result gives useful bounds on the free energy $\F(\rho^n)$ and $\F(\overline{\rho}^{n-1})$.

\begin{lemma}\label{lem:free-energy}
Suppose that $\F(\rho^0)$ is finite. Let $\{\rho^{n}\}_{n=1}^N$ be the sequence of minimizers given by
Scheme~\ref{sc: app scheme}. Then for any positive integer $1\leq n\leq N$,
\begin{align}
\F(\rho^n)&\leq \F(\rho^0)\quad \mbox{and}\quad \F(\overline{\rho}^{n-1})\leq \F(\rho^0).\label{eq:est31}
\end{align}
\end{lemma}
\begin{proof}
Since $\rho^{n}$ is the minimizer of
problem~\eqref{eq: minimization prob} and $\overline{\rho}^{n-1}$ is an admissible density, we have
\begin{equation*}
\frac{C_\alpha}{2\tau^\alpha}W_2^2(\rho^{n},\overline{\rho}^{n-1})+\F(\rho^{n})\leq \frac{C_\alpha}{2\tau^\alpha}W_2^2(\overline{\rho}^{n-1},\overline{\rho}^{n-1})+\F(\overline{\rho}^{n-1})=\F(\overline{\rho}^{n-1}),
\end{equation*}
which implies
\begin{equation}\label{eq: bound of W2}
W_2^2(\rho^{n},\overline{\rho}^{n-1})\leq 2C_\alpha^{-1}\tau^\alpha\big(\F(\overline{\rho}^{n-1})-\F(\rho^{n})\big).
\end{equation}
It follows from this inequality and Lemma \ref{lem:convex} that
$$
\F(\rho^n)\leq \F(\overline{\rho}^{n-1})\leq \sum_{i=0}^{n-1} (-b_{n-i}^{(n)})\F(\rho^i).
$$
Then by mathematical induction, we claim $\F(\rho^n)\leq \F(\rho^0)$.
Indeed, the claim holds trivially for $n=0$. Now suppose it holds up to $n\leq N-1$, then
by the induction hypothesis and the facts that $b_{n+1-i}^{(n+1)}< 0$ for $i=0,\ldots,k$ and
$\sum_{i=0}^n(-b_{n+1-i}^{(n+1)})=1$, cf. Lemma \ref{lem: properties of b}, it follows
\begin{equation*}
  \F(\rho^{n+1}) \leq \sum_{i=0}^{n}(-b_{n+1-i}^{(n+1)})\F(\rho^i) \leq \sum_{i=0}^n(-b_{n+1-i}^{(n+1)})\F(\rho^0) = \F(\rho^0),
\end{equation*}
which shows directly the first assertion. Then the second assertion follows immediately as
$\F(\overline{\rho}^{n-1})\leq \sum_{i=0}^{n-1} (-b_{n-i}^{(n)})\F(\rho^i)\leq \F(\rho^0)
\sum_{i=0}^{n-1} (-b_{n-i}^{(n)})\leq \F(\rho^0)$, cf. Lemma \ref{lem:convex}.
\end{proof}

The next result gives a uniform bound on the second moment $M_2(\rho^n)$ of the approximation $\rho^n$,
which plays a crucial role in the convergence analysis. The proof crucially employs the property of the
relative entropy. Recall that the relative entropy $\H(\mu,\nu)$ between two probability measures $\mu$
and $\nu$ is defined by
\[
\H(\mu|\nu)=\begin{cases}
\int \log\Big(\frac{d\mu}{d\nu}\Big)d\mu~~\text{if}~~d\mu\ll d\nu\\ +\infty~~\text{otherwises}.
\end{cases}
\]
By Jensen's inequality, $\H(\mu|\nu)\geq 0$ for all $\mu$ and $\nu$. Taking $\mu\in \mathcal{P}_2(\mathbb{R}^d)$ and $\nu=Z^{-1}e^{-\frac{|x|^2}{2m}}$,
where $Z=(2\pi m)^{-\frac d2}$ is the normalization constant ($m>0$ is to be chosen), gives
\[
0\leq \H(\mu|Z^{-1}e^{-\frac{|x|^2}{2m}})=\int_{\mathbb{R}^d}\mu\log\mu\d x +\frac{1}{2m}\int_{\mathbb{R}^d} |x|^2\mu\d x+\log Z,
\]
This implies the following useful inequality
\begin{equation}
\label{entropy-moment}
-\int_{\mathbb{R}^d}\mu\log\mu\d x\leq \tfrac{1}{2m}M_2(\mu)-\tfrac{d}{2} \log(2\pi m).
\end{equation}

\begin{lemma}\label{lem:moment}
Suppose that $\F(\rho^0)$ and $M_2(\rho^0)$ are finite. Let $\{\rho^{n}\}_{n=1}^N$ be the sequence of minimizers given by
Scheme~\ref{sc: app scheme}. Then for any positive integer $1\leq n\leq N$, there holds
\begin{align}
M_2(\rho^n)&\leq C.\label{eq: est1}
\end{align}
\end{lemma}
\begin{proof}
The proof of the lemma is inspired by \cite[Lemma 3.7, (3.10)]{EberleNiethammerSchlichting:2017}.
Let $P^n$ be the optimal plan for the Wasserstein distance $W_2(\rho^n,\overline{\rho}^{n-1})$ between
$\overline{\rho}^{n-1}$ and $\rho^n$. Then by the definition of the second moment $M_2(\rho^n)$, there holds
\begin{align*}
  M_2(\rho^n) & = \int|y|^2\rho^n(\d y) = \int |y|^2P^n(\d x\d y)\\
    & \leq \int (2|y-x|^2+2 x^2)P^n(\d x \d y)\\
    & = 2W_2^2(\rho^n,\overline{\rho}^{n-1}) + 2M_2(\overline{\rho}^{n-1})\\
    & = 2W_2^2(\rho^n,\overline{\rho}^{n-1}) + 2\sum_{i=0}^{n-1}(-b_i^{(n)})M_2(\rho^i).
\end{align*}
By means of mathematical induction, this estimate, the inequality \eqref{eq: bound of W2} and the assumptions $\F(\rho^0)<\infty$
and $M_2(\rho^0)<0$ directly imply that the second moment $M_2(\rho^n)$ of each of the approximation $\rho^n$ is indeed finite.
To derive a uniform bound (with respect to $n$ and $\tau$), we estimate the ``fractional'' difference
quotient of the second moment using Corollary \ref{cor:EL} with $\varphi=|x|^2$. This choice is justified
by the finiteness of the second moment of each of the $\rho^n$:
\begin{align*}
  \frac{C_\alpha}{\tau^\alpha}(M_2(\rho^n)-M_2(\overline{\rho}^{n-1})) &
  = \frac{C_\alpha}{\tau^\alpha}\int_{\mathbb{R}^d}x^2(\rho^n-\overline{\rho}^{n-1})\d x\\
  &\leq \Big|\int_{\mathbb{R}^d}(2\nabla\Psi\cdot x-2)\rho^n(x)\d x\Big| + \tau^{-\alpha}W_2^2(\rho^n,\overline{\rho}^{n-1})\\
  &\leq 2\int_{\mathbb{R}^d}|\nabla \Psi||x|\rho^n\d x + 2 + \tau^{-\alpha}W_2^2(\rho^n,\overline{\rho}^{n-1}).
\end{align*}
Now by the growth condition \eqref{ass:Psi} on $\Psi$, we have
\begin{equation*}
  \int_{\mathbb{R}^d}|\nabla \Psi||x|\rho^n\d x \leq C(1+M_2(\rho^n)).
\end{equation*}
It follows from these two estimates and \eqref{eq: bound of W2} that
\begin{align*}
  C_\alpha\tau^{-\alpha}(M_2(\rho^n)-M_2(\overline{\rho}^{n-1})) &\leq C(1+M_2(\rho^n))
  + \tau^{-\alpha}W_2^2(\rho^n,\overline{\rho}^{n-1})\\
  &\leq C(1+M_2(\rho^n)) + C_\alpha(\F(\overline{\rho}^{n-1})-\F(\rho^n)).
\end{align*}
Next we bound the terms $\F(\overline{\rho}^{n-1})$ and $-F(\rho^n)$ on the right hand side. First, by Lemma
\ref{lem:free-energy}, $\F(\overline{\rho}^{n-1})\leq \F(\rho^0)<\infty$.
Meanwhile, since $\Psi\geq 0$ by Assumption \ref{ass:Psi}, we obtain from \eqref{entropy-moment} that
\[
-\F(\mu)=-\int_{\mathbb{R}^d} \mu\log\mu \d x-\int_{\mathbb{R}^d}\Psi\mu\d x\leq -\int_{\mathbb{R}^d} \mu\log\mu \d x\leq \tfrac{1}{2m} M_2(\mu)-\tfrac{d}{2} \log(2\pi m).
\]
Applying this inequality with $\mu=\rho^n$ and $m=1/2$ gives
\begin{equation}\label{eqn:energy-Moment}
-\F(\rho^n)\leq M_2(\rho^n)-\frac{d}{2}\log(\pi).
\end{equation}
These estimates together imply
\begin{align*}
  C_\alpha\tau^{-\alpha}(M_2(\rho^n)-M_2(\overline{\rho}^{n-1})) &\leq C(1+M_2(\rho^n)).
\end{align*}
Further, by the definition of $\overline{\rho}^{n-1}$,
\begin{equation*}
  M_2(\overline{\rho}^{n-1}) = \int_{\mathbb{R}^d}|x|^2\overline{\rho}^{n-1}\d x = \int_{\mathbb{R}^d}|x|^2\sum_{i=0}^{n-1}(-b_{n-i}^{(n)})\rho^i\d x =\sum_{i=0}^{n-1}(-b_{n-i}^{(n)}M_2(\rho^i).
\end{equation*}
Together with the definition of the L1 scheme in \eqref{eqn:L1}, it implies
\begin{equation*}
  \bar \partial_\tau^\alpha M_2(\rho^n)\leq C(1+M_2(\rho^n)).
\end{equation*}
This and the discrete Gronwall's inequality from Lemma \ref{lem:Gronwall} immediately imply the desired assertion.
\end{proof}

The next result gives a uniform bound on the entropy and energy of the approximations $\{\rho^n\}_{n=1}^N$, which induces the necessary compactness
needed in the proof of Theorem \ref{thm: main theorem}. The notation $[]_+$ denotes taking the positive part.
\begin{lemma}\label{lem:entropy}
Suppose that $\F(\rho^0)$ and $M_2(\rho^0)$ are finite. Let $\{\rho^{n}\}_{n=1}^N$ be the sequence of minimizers given by
Scheme~\ref{sc: app scheme}. Then for any positive integer $1\leq n\leq N$, there hold
\begin{align*}
\int_{\R^d}[\rho^n \log\rho^n]_+\,\d x\leq C, \quad \E(\rho^n)\leq C ,\quad
\sum_{n=1}^k W_2^2(\rho^{n},\overline{\rho}^{n-1})\leq C\tau^\alpha. 
\end{align*}
\end{lemma}
\begin{proof}
These estimates are analogous to (43), (44) and (45) in \cite{JordanKinderlehrerOtto:1998}.
According to \cite[Equations (14)-(15)]{JordanKinderlehrerOtto:1998},
there exist $0<\gamma<1$ and $C<\infty$ such that for all $\rho\in\mathcal{P}_2(\mathbb{R}^d)$
\begin{equation}
\label{eq: S-M estimates}
S(\rho)\geq -C(M_2(\rho)+1)^\gamma\quad\text{and}\quad \int_{\R^d}|\min\{\rho\log\rho,0\}|\,dx \leq C(M_2(\rho)+1)^\gamma.
\end{equation}
Now the first two estimates follow directly from \eqref{eq: S-M estimates} and \eqref{eq: est1},
and Lemmas \ref{lem:free-energy} and \ref{lem:moment} as
\begin{align*}
\int \max\{\rho^n\log\rho^n,0\}\,dx&\leq S(\rho^n)+\int_{\R^d}|\min\{\rho^n\log\rho^N,0\}|\,dx
\\&\leq S(\rho^n)+C(M_2(\rho^n)+1)^\gamma
\\& \leq \F(\rho^n)+C(M_2(\rho^n)+1)^\gamma\leq C,\\
\E(\rho^k)&=\F(\rho^n)-\S(\rho^n)
\\&\leq \F(\rho^n)+C(M_2(\rho^n)+1)^\gamma \leq C.
\end{align*}
It remains to prove the last estimate. By summing \eqref{eq: bound of W2} over $n$ and using the inequality
\eqref{eq: bound sum F}, we obtain
\begin{align}\label{eq: bound sum W2}
&\quad \sum_{i=1}^{n}W_2^{2}(\rho^{i},\bar{\rho}^{i-1})\leq \frac{2\tau^\alpha}{C_\alpha}\sum_{i=1}^{n}\big(\F(\bar{\rho}^{i-1})-\F(\rho^{i})\big)\nonumber\\
&\leq \frac{2\tau^\alpha}{C_\alpha}\left[n^{1-\alpha}\F(\rho^0)+\sum_{i=1}^{n-1}\Big((n-i)^{1-\alpha}-(n-i+1)^{1-\alpha}\Big)\F(\rho^i)-\F(\rho^n)\right].
\end{align}
Next we bound the summation in the square bracket. The inequality \eqref{entropy-moment} (with $m$ to be chosen below) implies
\begin{align*}
&\quad\sum_{i=1}^{n-1}\Big((n-i)^{1-\alpha}-(n-i+1)^{1-\alpha}\Big)\F(\rho^i)\\
&=\sum_{i=1}^{n-1}\Big((n-i+1)^{1-\alpha}-(n-i)^{1-\alpha}\Big)(-\F(\rho^i))\\
&\leq \sum_{i=1}^{n-1}\big((n-i+1)^{1-\alpha}-(n-i)^{1-\alpha}\big)\Big(\frac{1}{2m}M_2(\rho^i)-\frac{d}{2}\log(2\pi m)\Big)\\
&\leq\frac{M}{2m}\sum_{i=1}^{n-1}\big((n-i+1)^{1-\alpha}-(n-i)^{1-\alpha}\big)-\frac{d}{2}\log(2\pi m)(n^{1-\alpha}-1)\\
&=n^{1-\alpha}\Big[\frac{M}{2m}-\frac{d}{2}\log(2\pi m)\Big]+\frac{d}{2}\log(2\pi m)-\frac{M}{2m},
\end{align*}
where $M>0$ is an upper bound of $M_2(\rho^i)$ for all $i=1,\ldots,n-1$ derived in Lemma \ref{lem:moment}. Consequently,
\begin{align}\label{bound sumF}
&n^{1-\alpha}\F(\rho^0)+\sum_{i=1}^{n-1}\Big((n-i)^{1-\alpha}-(n-i+1)^{1-\alpha}\Big)\F(\rho^i)\notag\\
\leq&n^{1-\alpha}\Big[\F(\rho^0)+\frac{M}{2m}-\frac{d}{2}\log(2\pi m)\Big]+\frac{d}{2}\log(2\pi m)-\frac{M}{2m}.
\end{align}
It suffices to bound the right hand side uniformly with respect to $n$. To this end,
let $g:(0,+\infty)\to \mathbb{R}$ be defined by $g(m):=\F(\rho^0)+\frac{M}{2m}-\frac{d}{2}\log(2\pi m)$. Then, simple
computation shows $g'(m)=-\frac{M}{2m^2}-\frac{d}{2m}<0$, $\lim_{m\rightarrow 0^+}
g(m)=+\infty$ and $\lim_{m\rightarrow +\infty}g(m)=-\infty$. Thus, the equation $g(m)=0$ has a unique solution
$m^*\in (0,+\infty)$ that depends only on $\F(\rho^0)$, $d$ and $M$. Choosing $m=m^*$ in \eqref{bound sumF} gives
$$
n^{1-\alpha}\F(\rho^0)+\sum_{i=1}^{n-1}\Big((n-i)^{1-\alpha}-(n-i+1)^{1-\alpha}\Big)\F(\rho^i)\leq \frac{d}{2}\log(2\pi m^*)-\frac{M}{2m^*}=\F(\rho^0).
$$
This estimate, \eqref{eq: bound sum W2} and \eqref{eqn:energy-Moment} together imply
\begin{align*}
\sum_{i=1}^n W^2_2(\rho^i,\bar{\rho}^{i-1})&\leq C\tau^\alpha \left[ n^{1-\alpha}\F(\rho^0)+\sum_{i=1}^{n-1}\Big((n-i)^{1-\alpha}-(n-i+1)^{1-\alpha}\Big)\F(\rho^i)-\F(\rho^n)\right]\\
&\leq C\tau^\alpha\Big[\F(\rho^0)+M(\rho^n)-\frac{d}{2}\log(\pi)\Big] \leq C\tau^\alpha,
\end{align*}
where the last step follows from Lemma \ref{lem:moment}. This completes the proof of the lemma.
\end{proof}

Now we can state the proof of Theorem \ref{thm: main theorem}.
\begin{proof}[Proof of Theorem \ref{thm: main theorem}]
The proof follows the strategy in~\cite{JordanKinderlehrerOtto:1998,Huang:2000,DuongPeletierZimmer:2014}.
The key idea is to pass to the limit
$\tau\to 0^+$ in the Euler-Lagrange equation for the sequence of minimizers~\eqref{eq: EL eqn}
proved in Lemma~\ref{lem: EL eqn}. The \textit{a priori} estimates in
Lemmas~\ref{lem:moment} and \ref{lem:entropy} provide necessary compactness properties that allow
us to extract a convergent subsequence.

Let $T>0$ be a given final time. For each fixed $\tau>0$, let $\{\rho^n\}_{n=1}^N$ be the sequence
of minimizers given by Scheme \ref{sc: app scheme} and let $t\mapsto \rho_\tau(t)$ be the
approximation defined in~\eqref{eq: interpolation}. By Lemmas~\ref{lem:moment} and \ref{lem:entropy}, we have
\begin{align}
M_2(\rho_\tau(t))+\int_{\mathbb{R}^{d}}[\rho_\tau(t)\log \rho_\tau(t)]^+\d x&\leq C,\qquad\text{for all} \quad 0\leq t \leq T. \label{sumM2Entropy}
\end{align}
Since the function $z\mapsto [z\log z]^+$ has super-linear growth, the bound \eqref{sumM2Entropy} and
Dunford-Pettis theorem \cite{Yosida:1980} ensure that there exists a subsequence, denoted again by $\rho_\tau$,
and some $\rho\in L^1((0,T)\times\R^{d})$ such that
\begin{equation}
\rho_\tau\rightarrow \rho ~\text{ weakly in }~ L^1((0,T)\times\R^{d}).
\end{equation}
It remains to show that the limit $\rho$ satisfies the weak formulation~\eqref{eq: weak formulation}
of problem~\eqref{eq:fractionalFP} in the sense of Definition \ref{def: weak sol}. Fix any test
function $\varphi\in C_c^\infty((-\infty,T)\times \R^{d})$. Let $P^n\in\P(\overline{\rho}^{n-1},\rho^n)$ be the
optimal plan for $W_2(\overline{\rho}^{n-1},\rho^n)$. For any $0<t<T$, we have
\begin{align}
&\quad\int_{\mathbb{R}^{d}}\big[\rho^n(x)-\overline{\rho}^{n-1}(x)\big]\,\varphi(t,x)\d x\nonumber\\
&=\int_{\mathbb{R}^{d}}\rho^n(y)\varphi(t,y)\d y-\int_{\mathbb{R}^{d}}\overline{\rho}^{n-1}(x)\varphi(t,x)\d x\nonumber\\
&=\int_{\mathbb{R}^{2d}}\big[\varphi(t,y)-\varphi(t,x)\big]\,P^n(\d x \d y) \nonumber
\\&=\int_{\mathbb{R}^{2d}}(y-x)\cdot\nabla \varphi(t,y)P^n(\d x\d y) +\varepsilon_n,\label{timederivativeappro}
\end{align}
where in the last line, we have used Taylor expansion of $\varphi$. The error term $\varepsilon_n$ depends
on $t$ through time-dependence of $\varphi$ and can be bounded by
\begin{align}\label{epsilon}
|\varepsilon_n(t)|&\leq C\int_{\mathbb{R}^{2d}}|y-x|^2\,P^n(\d x\d y)\leq C W_2^2(\overline{\rho}^{n-1},\rho^n).
\end{align}
From Lemma \ref{lem: EL eqn} and the identity \eqref{timederivativeappro}, we obtain
\begin{equation}
\frac{C_\alpha}{\tau^\alpha}\int_{\mathbb{R}^{d}}\big[\rho^n(x)-\overline{\rho}^{n-1}(x)\big]\,\varphi(t,x)\d x=\int_{\mathbb{R}^d}(-\nabla\Psi\cdot\nabla\varphi+\Delta\varphi)\rho^n(x)\,\d x-\frac{C_\alpha}{\tau^\alpha}\varepsilon_n,
\end{equation}
which, upon integrating with respect to $t$ from $t_{n-1}$ to $t_n$, yields
\begin{align*}
&\quad\frac{C_\alpha}{\tau^\alpha}\int_{t_{n-1}}^{t_n}\int_{\mathbb{R}^d}[\rho^n(x)-\overline{\rho}^{n-1}(x)]\varphi(t,x)\,\d x\,\d t\nonumber\\
&=\int_{t_{n-1}}^{t_n}\int_{\mathbb{R}^d}(-\nabla\Psi\cdot\nabla\varphi+\Delta\varphi)\,\rho^n(x)\,\d x\d t-\frac{C_\alpha}{\tau^\alpha}\int_{t_{n-1}}^{t_n}\varepsilon_n\,\d t\notag
\\&=\int_{t_{n-1}}^{t_n}\int_{\mathbb{R}^d}(-\nabla\Psi\cdot\nabla\varphi+\Delta\varphi)\,\rho_\tau(t,x)\,\d x\,\d t-\frac{C_\alpha}{\tau^\alpha}\int_{t_{n-1}}^{t_n}\varepsilon_n\,\d t,
\end{align*}
where the last line follows from the definition of the piecewise constant interpolation $\rho_\tau(t,x)$.
Summing the last identity from $n=1$ to $N$ gives
\begin{multline}\label{eq: sum 1}
\sum_{n=1}^{N}\frac{C_\alpha}{\tau^\alpha}\int_{t_{n-1}}^{t_n}\int_{\mathbb{R}^d}[\rho^n(x)-\overline{\rho}^{n-1}(x)]\varphi(t,x)\,\d x\,\d t\\
=\int_0^T\int_{\mathbb{R}^d}(-\nabla\Psi\cdot\nabla\varphi+\Delta\varphi)\,\rho_\tau(t,x)\,\d x\,\d t+e_\tau,
\end{multline}
where the term $e_\tau$ is given by
\begin{equation}
\label{eq: error1}
e_\tau=-\frac{C_\alpha}{\tau^\alpha}\sum_{n=1}^{N}\int_{t_{n-1}}^{t_n}\varepsilon_n\,\d t.
\end{equation}
Now recall that by the definition $\overline\rho^{n-1}$ of Scheme \ref{sc: app scheme},
$\overline{\rho}^{n-1}=\sum_{i=0}^{n-1}(-b_{n-i}^{(n)})\rho^i$. This and the definition of the L1
approximation in \eqref{eqn:L1}, we can rewrite the left hand side of the identity \eqref{eq: sum 1} as
\begin{align*}
&\sum_{n=1}^{N}\frac{C_\alpha}{\tau^\alpha}\int_{t_{n-1}}^{t_n}\int_{\mathbb{R}^d}[\rho^n(x)-\overline{\rho}^{n-1}(x)]\varphi(t,x)\,\d x\,\d t\\
=&\sum_{n=1}^{N}\frac{C_\alpha}{\tau^\alpha}\int_{t_{n-1}}^{t_n}\int_{\mathbb{R}^d}\Big[\rho^n(x)+\sum_{i=0}^{n-1}b_{n-i}^{(n)}\rho^i(x)\Big]\varphi(t,x)\,\d x\,\d t\\
=& \sum_{n=1}^N\int_{t_{n-1}}^{t_n}\int_{\mathbb{R}^d}\bar\partial_\tau^\alpha \rho^n(x)\varphi(t,x)\,\d x\,\d t.
\end{align*}
By Lemma \ref{lem:int-by-part-semi}, we obtain
\begin{align*}
  &\sum_{n=1}^N \int_{t_{n-1}}^{t_n}\int_{\mathbb{R}^d}(\bar \partial_\tau^\alpha \rho^n)(t,x) \varphi(t,x)\d t\d x\\
  =& \int_{\mathbb{R}^d}\int_0^T \rho_\tau(t,x) {\overline{D}_\tau^\alpha \varphi(t,x)}\d t\d x  + C_\alpha\tau^{-\alpha}\int_{\mathbb{R}^d}\rho^0(x)\Big(\sum_{n=1}^Nb_n^{(n)}\int_{t_{n-1}}^{t_n} \varphi(t,x)\d t\Big)\d x.
\end{align*}
Now by Theorem \ref{thm:err-approx}, the following two limits hold
\begin{align*}
  \lim_{\tau\to 0^+} \overline{D}_\tau^\alpha\varphi (t) & = {_tD_T^\alpha}\varphi(t),\\
  \lim_{\tau\to 0^+} C_\alpha\tau^{-\alpha}\sum_{n=1}^N b_n^{(n)} \int_{t_{n-1}}^{t_n} \varphi(t)\d t &= - ({_tI_T^\alpha}\varphi)(0) .
\end{align*}
Thus, upon passing to limit $\tau\to0^+$ and noting the weak convergence of the sequence $\rho_\tau$ to $\rho$ in $L^1(\Omega)$,
we deduce
\begin{align*}
  &\lim_{\tau\to0^+} \sum_{n=1}^{N}\frac{C_\alpha}{\tau^\alpha}\int_{t_{n-1}}^{t_n}\int_{\mathbb{R}^d}[\rho^n(x)-\bar{\rho}^{n-1}(x)]\varphi(t,x)\,\d x\,\d t\\
  =& \lim_{\tau\to0^+} \int_{\mathbb{R}^d}\int_0^T \rho_\tau(t,x) {\overline{D}_\tau^\alpha \varphi(t,x)}\d t\d x  + C_\alpha\tau^{-\alpha}\int_{\mathbb{R}^d}\rho^0(x)\Big(\sum_{n=1}^Nb_n^{(n)}\int_{t_{n-1}}^{t_n} \varphi(t,x)\d t\Big)\d x\\
  =& \int_{\mathbb{R}^d}\int_0^T \rho(t,x) {_tD_T^\alpha \varphi(t,x)}\d t\d x  - \frac{1}{\Gamma(1-\alpha)} \int_{\mathbb{R}^d}\rho^0(x)\int_0^T t^{-\alpha}\varphi(t,x)\d t\d x.
\end{align*}
Meanwhile, for the first term on the right-hand side of \eqref{eq: sum 1}, using the weak convergence of $\rho_\tau$ to $\rho$ in
$L^1((0,T)\times\mathbb{R}^d)$, we obtain
$$
\lim_{\tau\to0^+}\int_0^T\int_{\mathbb{R}^d}(-\nabla\Psi\cdot\nabla\varphi+\Delta\varphi)\,\rho_\tau(t,x)\,\d x\d t=\int_0^T\int_{\mathbb{R}^d}(-\nabla\Psi\cdot\nabla\varphi+\Delta\varphi)\,\rho(t,x)\,\d x\d t.
$$
It remains to consider the error term $e_\tau$. Actually, by the estimates \eqref{eq: error1} and \eqref{epsilon}, we have
\begin{align*}
|e_\tau|&\leq\frac{C_\alpha}{\tau^\alpha}\sum_{n=1}^{N}\int_{t_{n-1}}^{t_n}|\varepsilon_n(t)|\,\d t{\leq}C\tau^{-\alpha} \sum_{n=1}^{N}\int_{t_{n-1}}^{t_n}W_2^2(\bar{\rho}^{n-1},\rho^n)\,\d t\\
&\leq C\tau^{1-\alpha}\sum_{n=1}^NW_2^2(\bar{\rho}^{n-1},\rho^n)\leq C\tau,
\end{align*}
where the last step is due to the bound on $\sum_{n=1}^NW_2^2(\bar{\rho}^{n-1},\rho^n)$ from Lemma \ref{lem:entropy}.
This inequality implies that $e_\tau\rightarrow 0$ as $\tau\rightarrow 0^+$. Therefore, taking the
limit $\tau\rightarrow 0^+$, we deduce that the limiting density $\rho(t,x)$
satisfies
\begin{align*}
&\int_0^T\int_{\mathbb{R}^d}\rho(t,x){_tD_T^\alpha}\varphi(t,x)\d t\d x-\frac{1}{\Gamma(1-\alpha)}\int_0^Tt^{-\alpha}\int_{\mathbb{R}^d}\rho_0(x)\varphi(t,x)\d x\,\d t\\
=&\int_0^T\int_{\mathbb{R}^d}(-\nabla\Psi\cdot\nabla\varphi+\Delta\varphi)\,\rho(t,x)\,\d x\,\d t
\end{align*}
which is precisely the weak formulation \eqref{eq: weak formulation} of \eqref{eq:fractionalFP} in Definition
\ref{def: weak sol}. This completes the proof of the theorem.
\end{proof}

\section{Numerical experiments}\label{sec:numer}

The classical JKO scheme may be employed as a time-stepping scheme for solving FPEs \cite{KinderlehrerWalkington1999,
AguelBowles2013}, although not extensively studied due to relatively high computational cost associated with the
Wasserstein distance. Following the setting in \cite{KinderlehrerWalkington1999}, we illustrate the fractional scheme
\ref{sc: app scheme} with the following time-fractional FPE:
\begin{equation}
\label{eq: 1D model}
\partial_t^\alpha u-\nabla\cdot(\nabla u+ \nabla \Psi u)=0\quad\text{in}~~\Omega,
\end{equation}
subject to the following initial and boundary conditions
\begin{equation}\label{eq: IBCs}
u\big\vert_{t=0}=u_0\geq 0\quad\text{and}\quad (\nabla u+\nabla \Psi u)\cdot \nu =0~~\text{on}~~\partial\Omega,
\end{equation}
where $\nu$ is the unit outward normal direction, and $u_0$ is assumed to be a probability density in $\Omega$, i.e., $\int_{\Omega}u_0(x)\d x=1$.

\subsection{Implementation details}
Based on Scheme \ref{sc: app scheme}, we employ the following time semi-discrete approximation: Let $u^0:=u_0$, and
for $n\geq 1$, define $u^{n}$ to be the unique minimizer over $\mathcal{A}:=\big\{u:\Omega\rightarrow\infty:
u\in L^1(\Omega)~\text{and}~\int_{\Omega}u(x)\d x=\int_{\Omega} u^{n}(x) \d x\big\}$ of the following functional
\begin{equation}\label{eq: I}
J(u):=\frac{C_\alpha}{2\tau^\alpha} W_2^2(u,\overline{u}^{n-1})+\int_{\Omega} (u\log u+ \Psi u)\d x,\quad \mbox{with } \overline{u}^{n-1}=\sum_{i=0}^{n-1}(-b^{(n)}_{n-i})u^i.
\end{equation}

Next we describe a spatial discretization of the function $J(u)$ for the one-dimensional case
$\Omega=(0,1)$, and the discretization is similarly for the high-dimensional case, provided that
one can have a regular decomposition (e.g., triangulation) of the domain $\Omega$. The interval $\Omega=[0,1]$ is discretized into
subintervals $[x_i, x_{i+1}]$, where $x_i=ih$ and $h=1/M$ denotes the mesh size and $M\in \mathbb{N}_+$. Similarly,
the time interval $[0,T]$ is discretized as into $[t_n,t_{n+1}]$, where $t_n=n\tau$, $n=0,1,\ldots,$ and $\tau=T/N$
denotes the time-step size and $M\in\mathbb{N}_+$. Following \cite{KinderlehrerWalkington1999,AguelBowles2013},
we approximate the solutions $u^n$ with spatially piecewise constant functions. The initial data $u_0:\Omega\rightarrow \mathbb{R}$ and the forcing $\Psi:
\Omega\rightarrow \mathbb{R}$ are taken to be piecewise constant functions whose values coincide with their
function values at the midpoint $x_{i+\frac12}=x_i+\frac{h}{2}$, i.e., $u_0$ by the sequence $\mathbf{u}^0=(u_{0i})_{i=0}^{M-1}\in\mathbb{R}^M$,
with $u_{0i}=u_0(x_{i+\frac12})$, and similarly for $\Psi$. Accordingly, the integral of a function
$f:\Omega\rightarrow \mathbb{R}$ over the domain $\Omega$ is approximate by the following midpoint quadrature:
\begin{equation*}
\int_{\Omega}f(x)\,\d x\simeq h\sum_{i=0}^{M-1}f_i.
\end{equation*}
By absorbing the mesh size $h$ into the probability, i.e., $u_i=\int_{x_{i}}^{x_{i+1}}u\d x$, then $\mathbf{u}\in\mathbb{R}^M$
belongs to the discrete probability space. The discrete analogue $J_h(\mathbf{u})$ of the functional $J(u)$ is given by
\begin{equation}\label{eqn:JKO-dis}
    J_h(\mathbf{u})=\frac{C_\alpha}{2\tau^\alpha} W_2^2(\mathbf{u},\mathbf{\overline{u}}^{n-1})+  \langle\mathbf{u},\mathbf{\log u}+\boldsymbol{\psi}\rangle,\quad \mbox{with } \overline{\mathbf{u}}^{n-1}=\sum_{i=0}^{n-1}(-b^{(n)}_{n-i})\mathbf{u}^i,
\end{equation}
where $\log$ (and exponential) of a vector is understood componentwise, $\langle\cdot,\cdot\rangle$ denotes the
usual Euclidean inner product on $\mathbb{R}^M$ (or $\mathbb{R}^{M\times M}$). The minimization is over the probability simplex
$\Sigma_M=\{\mathbf{u}\in\mathbb{R}^M: u_i\geq 0, \sum_{i=0}^{M-1}u_i=1\}$. In the functional, we
have dropped the constant independent of $\log \mathbf{u}$, since it does not affect the minimization.

The discrete functional $J_h(\mathbf{u})$ involves the Wasserstein distance $W_2^2(\mathbf{u},\mathbf{\overline{u}}^{n-1})$,
and thus its efficient minimization is nontrivial, which has restricted the computation of Wasserstein gradient flow
traditionally to the one spatial dimensional case, for which the Wasserstein distance can be computed explicitly via an
inverse cumulative distribution function \cite{KinderlehrerWalkington1999}. Nonetheless, over the past few years, the
computation of Wasserstein distance has witnessed significant progress, especially within the computer vision and
machine learning communities; see the monograph \cite{PeyreCuturi:2019} for an up-to-date account. In the numerical
experiments below, we employ the Dykstra algorithm given in Peyre \cite{Peyre:2015} for each JKO time-stepping. It is based on entropic approximation
of Wasserstein distance \cite{Cuturi:2013}, and easily extended to the multi-dimensional case, when compared with the
relaxation algorithm and projected gradient descent employed in \cite{KinderlehrerWalkington1999,AguelBowles2013}. We describe the
whole computational procedure for minimizing $J_h(\mathbf{u})$ in Appendix \ref{app:JKO-stepping} for the convenience of readers.
Note that the Wasserstein distance can also be approximated using the entropic regularization, leading to Sinkhorn algorithm
\cite{Cuturi:2013}. This algorithm is employed below to compute the error in Wasserstein distance approximately. In the computation, the
(crucial) relaxation parameter $\gamma$ in the algorithms is fixed at $1/N$, where $N$ is the number of time steps.

\subsection{Numerical results and discussions}

Now we present some numerical results. First we consider the one-dimensional case.
\begin{example}
The domain $\Omega=(0,1)$, the initial condition $u_0(x)=1$
in $\Omega$, and the forcing $\Psi$ is given by $\Psi(x)=x$ or $\Psi(x)=\frac12x^2$.
\end{example}

The numerical results are given in Tables \ref{tab:1d-x} and \ref{tab:1d-x2}, respectively, at the time $T=1$,
for the forcing $\Psi(x)=x$ and $\Psi(x)=\frac12x^2$, where the $L^1(\Omega)$ and $L^2(\Omega)$ error of the
numerical solutions with respect to the reference solution, which is computed on a much finer temporal grid
with a time step size $\tau=1/1280$. Note that the $L^1(\Omega)$ metric was employed in the prior studies
\cite{KinderlehrerWalkington1999,AguelBowles2013}, whereas the $L^2(\Omega)$ metric is very common in numerical
analysis \cite{JinLazarovZhou:2019}. In addition, we also present the error in the Wasserstein distance (indicated by $W$ in the tables),
computed using Sinkhorn algorithm \cite{Cuturi:2013}. The results show that the scheme is convergent in either norm, with
the convergence rate in the $L^1$ norm slightly higher than that for the $L^2$ norm. The convergence rate
is consistently observed to be sublinear for all fractional orders, and it is slower than the first-order
convergence of the standard implementation of the L1 scheme (implemented with the Galerkin in space)
\cite{JinLazarovZhou:2016ima}. Surprisingly, the convergence rate deteriorates as the fractional order $\alpha$
increases, however, the precise mechanism of the loss remains elusive. In sharp contrast, the convergence
in Wasserstein distance is rather stable with respect to the fractional order $\alpha$. The empirical rate is at
$0.47$, which is slower than the optimal first-order rate for the classical JKO scheme (under suitable conditions)
\cite[Theorem 4.0.4]{AmbrosioGigliSavare:2008}; see also \cite{ClementDesch:2010} and \cite[Theorem 2.7]{Craig:2016}
for a convergence rate $O(\tau^\frac14)$ and $O(\tau^\frac12)$, respectively. In view of these empirical observations, it is
of enormous interest to rigorously derive sharp convergence rate in the fractional case. Note that for the two forcing terms, the convergence
behavior of the scheme is very similar to each other; see Fig. \ref{fig:density-t=1} for the density function at
$T=1$. Qualitatively, the plots also indicate that the convergence speed to the equilibrium differs significantly
with the fractional order $\alpha$, as recently established by Kemppainen and Zacher \cite{Kemppainen2019}, i.e.,
the smaller is the fractional order $\alpha$, the slower is the convergence to the equilibrium.

\begin{table}[hbt!]
\centering
\caption{Numerical results for the forcing $\Psi(x)=x$.\label{tab:1d-x}}
\begin{tabular}{r|cccccc}
\hline
 $\alpha\backslash N$& 20 & 40 & 80 & 160 & 320 & rate \\
\hline
      $L^1$ & 2.22e-2 &  1.41e-2 &  8.85e-3 &  5.40e-3 &  2.98e-3  &  0.72\\
$0.6$ $L^2$ & 3.22e-2 &  2.18e-2 &  1.41e-2 &  8.62e-3 &  4.71e-3  &  0.69\\
     $W$   & 1.44e-1 &  1.05e-1 &  7.57e-2 &  5.42e-2 &  3.86e-2  &  0.47\\
\hline
      $L^1$ & 3.11e-2 &  2.18e-2 &  1.44e-2 &  9.42e-3 &  5.56e-3  &  0.62\\
$0.8$ $L^2$ & 4.43e-2 &  3.32e-2 &  2.38e-2 &  1.59e-2 &  9.52e-3  &  0.55\\
     $W$   & 1.44e-1 &  1.05e-1 &  7.58e-2 &  5.42e-2 &  3.87e-2  &  0.47\\
\hline
      $L^1$ & 3.37e-2 &  2.58e-2 &  1.83e-2 &  1.19e-2 &  6.81e-3  &  0.57\\
$1.0$ $L^2$ & 4.45e-2 &  3.66e-2 &  2.86e-2 &  2.07e-2 &  1.32e-2  &  0.43\\
     $W$   & 1.44e-1 &  1.05e-1 &  7.58e-2 &  5.43e-2 &  3.87e-2  &  0.47\\
\hline
\end{tabular}
\end{table}

\begin{table}[hbt!]
\centering
\caption{Numerical results for the forcing $\Psi(x)=x^2/2$.\label{tab:1d-x2}}
\begin{tabular}{r|cccccc}
\hline
$\alpha\backslash N$& 20 & 40 & 80 & 160 & 320 & rate \\
\hline
      $L^1$ & 1.41e-2  & 8.74e-3  & 5.06e-3  & 2.98e-3  & 1.61e-3  & 0.78\\
$0.6$ $L^2$ & 1.93e-2  & 1.30e-2  & 8.46e-3  & 5.14e-3  & 2.79e-3  & 0.69\\
  $W$   & 1.45e-1  & 1.05e-1  & 7.60e-2  & 5.44e-2  & 3.88e-2  & 0.47\\
\hline
      $L^1$ & 2.18e-2  & 1.51e-2  & 9.81e-3  & 5.85e-3  & 3.28e-3  & 0.68\\
$0.8$ $L^2$ & 2.88e-2  & 2.17e-2  & 1.56e-2  & 1.04e-2  & 6.22e-3  & 0.55\\
    $W$   & 1.45e-1  & 1.05e-1  & 7.60e-2  & 5.44e-2  & 3.87e-2  & 0.47\\
\hline
      $L^1$ & 3.04e-2  & 2.34e-2  & 1.66e-2  & 1.07e-2  & 6.09e-3  & 0.58\\
$1.0$ $L^2$ & 3.81e-2  & 3.15e-2  & 2.46e-2  & 1.78e-2  & 1.14e-2  & 0.43\\
     $W$   & 1.45e-1  & 1.05e-1  & 7.60e-2  & 5.44e-2  & 3.88e-2  & 0.47\\
\hline
\end{tabular}
\end{table}

\begin{figure}[hbt!]
  \centering
  \begin{tabular}{cc}
  \includegraphics[width=0.40\textwidth]{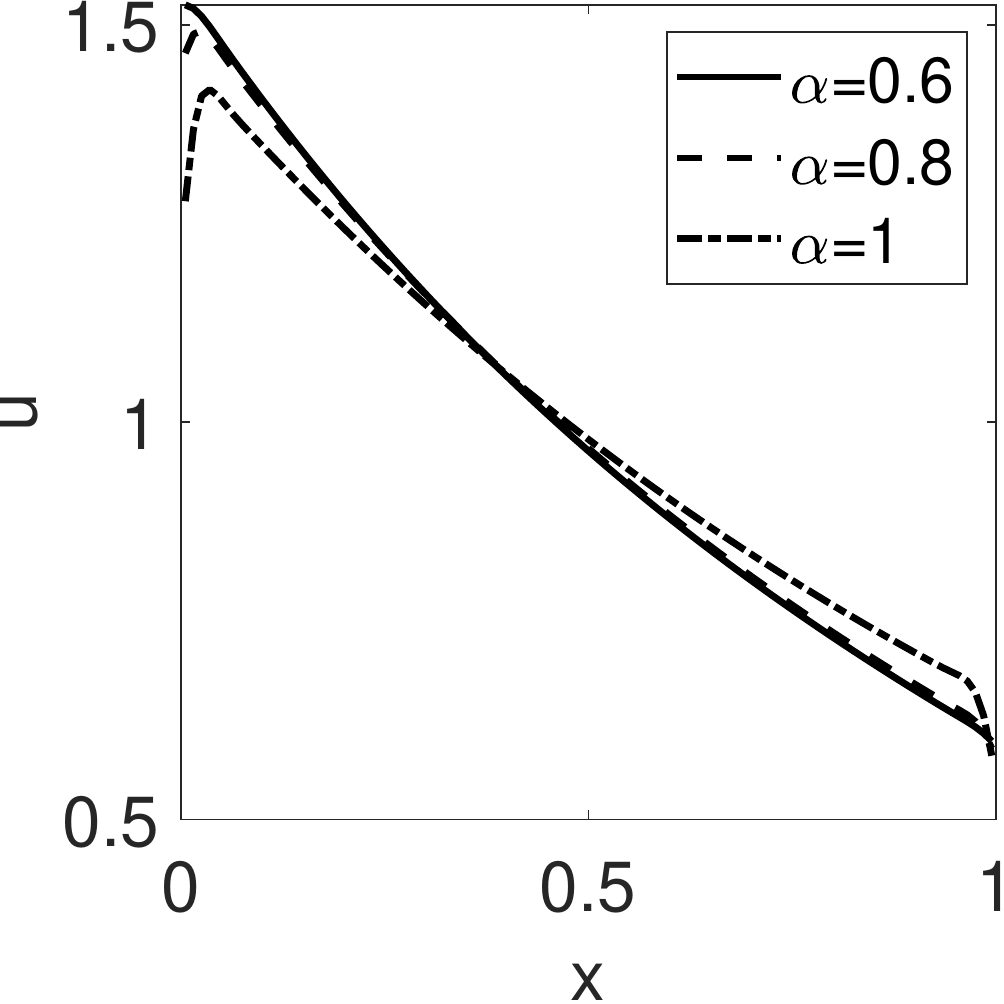} & \includegraphics[width=0.40\textwidth]{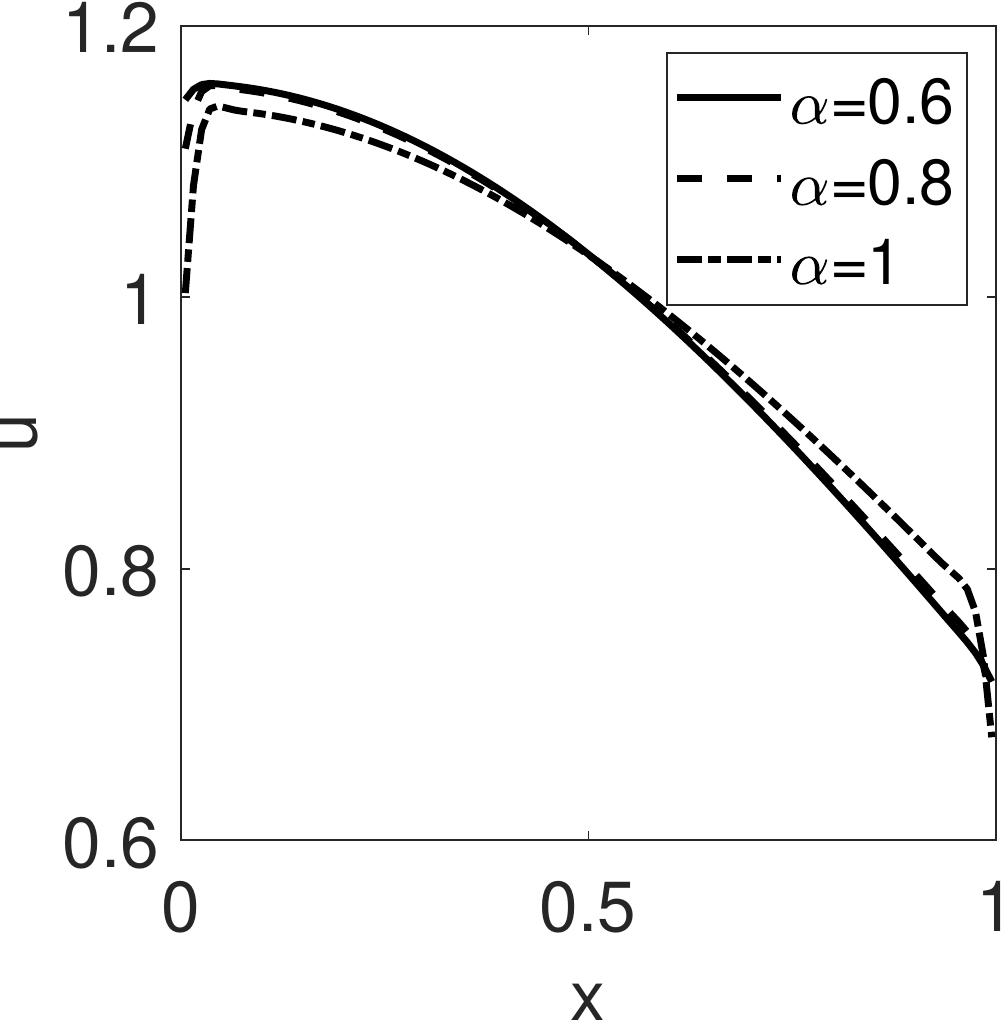}\\
   $\Psi(x) = x$  & $\Psi(x) =x^2/2$
  \end{tabular}
  \caption{The probability density function at $t=1$.\label{fig:density-t=1}}
\end{figure}

The next example is concerned with a two-dimensional problem.
\begin{example}\label{exam:2d}
The domain $\Omega=(0,1)^2$, the initial data $u_0(x)=1$ in $\Omega$, and the forcing $\Psi$ is given by $\Psi(x_1,x_2)=x_1+x_2$.
\end{example}

The numerical results for Example \ref{exam:2d} are presented in Table \ref{tab:2d} and Fig.
\ref{fig:density-2d-t=1}. The empirical convergence rates are similar to the one-dimensional
case, and the convergence is also very steady (but the computing time is much higher). The
density profiles for the three fractional orders are largely comparable at $T=1$, with the
main differences lie at the boundary, as observed in the one-dimensional case, cf. Fig. \ref{fig:density-t=1}. This is possibly
due to the difference in the long-time behavior for different fractional orders.
\begin{table}[hbt!]
\centering
\caption{Numerical results for the forcing $\Psi(x)=x_1+x_2$.\label{tab:2d}}
\begin{tabular}{r|cccccc}
\hline
$\alpha\backslash N$& 20 & 40 & 80 & 160 & 320 & rate \\
\hline
      $L^1$ & 3.28e-2  & 2.15e-2  & 1.35e-2  & 8.14e-3  & 4.43e-3  &  0.72\\
$0.6$ $L^2$ & 4.81e-2  & 3.27e-2  & 2.13e-2  & 1.31e-2  & 7.24e-3  &  0.68\\
     $W$   & 2.04e-1  & 1.48e-1  & 1.07e-1  & 7.66e-2  & 5.47e-2  &  0.47\\
\hline
      $L^1$ & 4.57e-2  & 3.31e-2  & 2.27e-2  & 1.46e-2  & 8.58e-3  &  0.60\\
$0.8$ $L^2$ & 6.63e-2  & 5.01e-2  & 3.62e-2  & 2.46e-2  & 1.49e-2  &  0.53\\
   $W$   & 2.04e-1  & 1.48e-1  & 1.07e-1  & 7.67e-2  & 5.47e-2  &  0.47\\
\hline
      $L^1$ & 4.93e-2  & 3.96e-2  & 2.92e-2  & 1.94e-2  & 1.09e-2  &  0.54\\
$1.0$ $L^2$ & 6.57e-2  & 5.43e-2  & 4.27e-2  & 3.13e-2  & 2.05e-2  &  0.42\\
      $W$   & 2.04e-1  & 1.49e-1  & 1.07e-1  & 7.68e-2  & 5.48e-2  &  0.47\\
\hline
\end{tabular}
\end{table}

\begin{figure}
  \centering
  \begin{tabular}{cccccc}
  \includegraphics[width=0.3\textwidth]{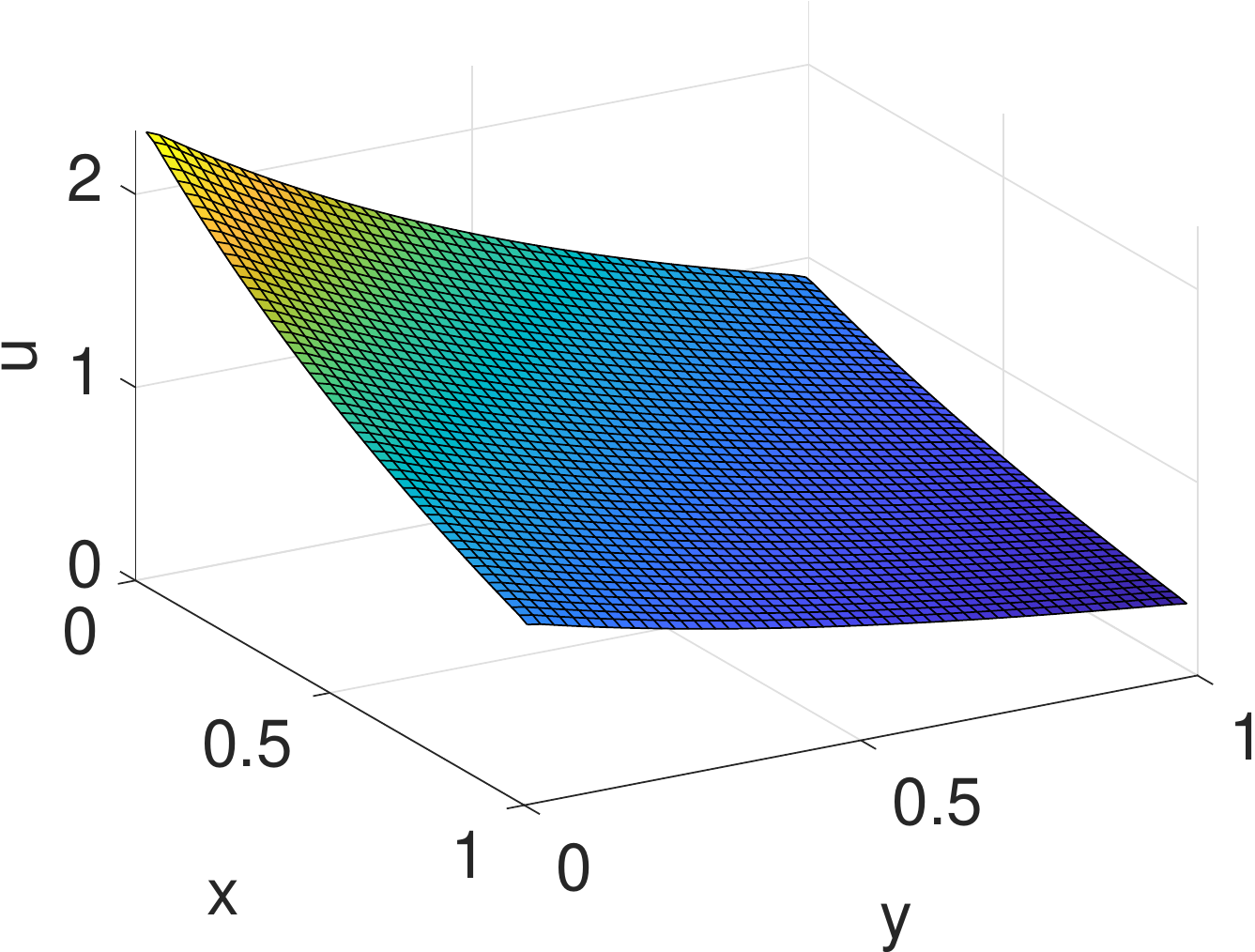} & \includegraphics[width=0.3\textwidth]{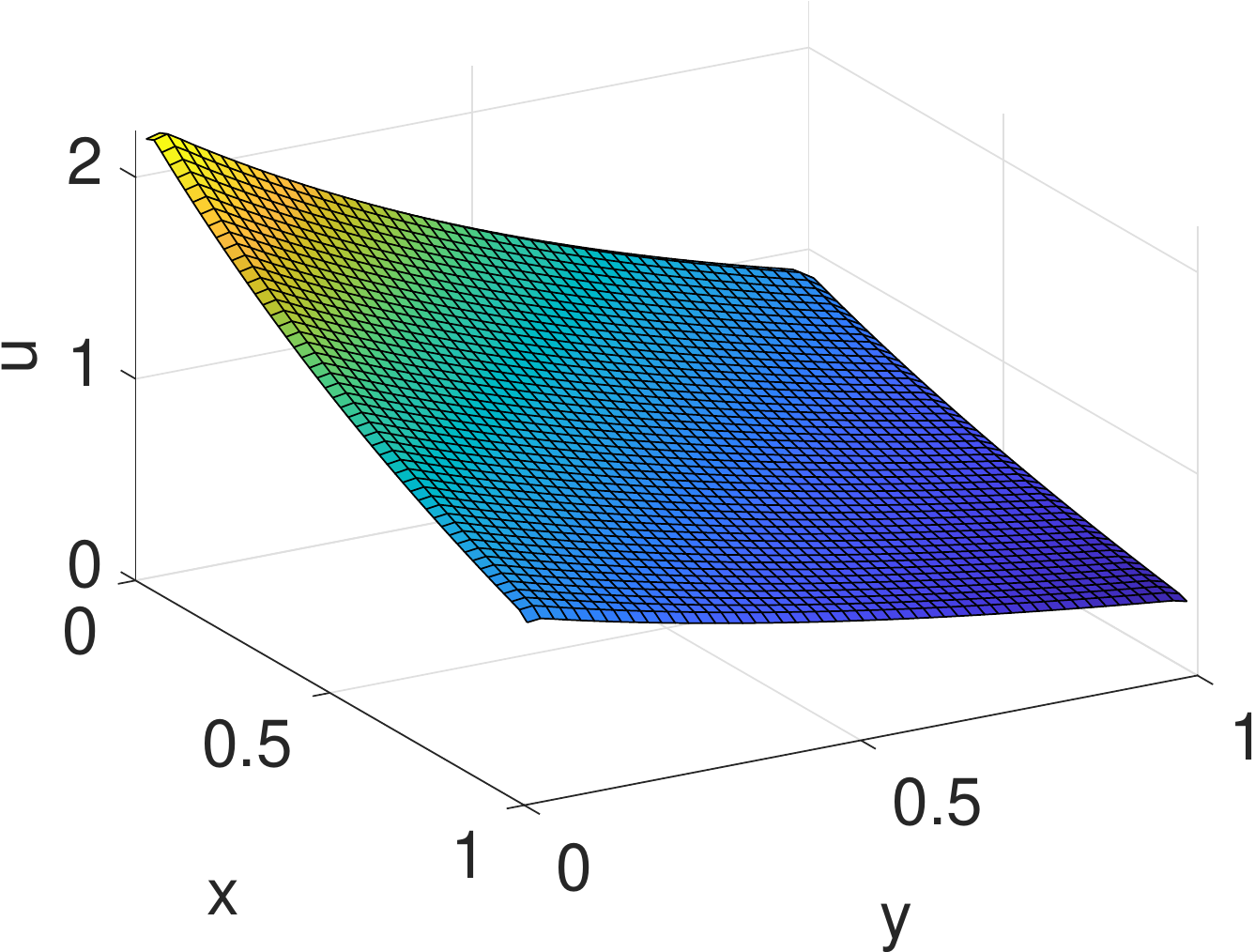} & \includegraphics[width=0.3\textwidth]{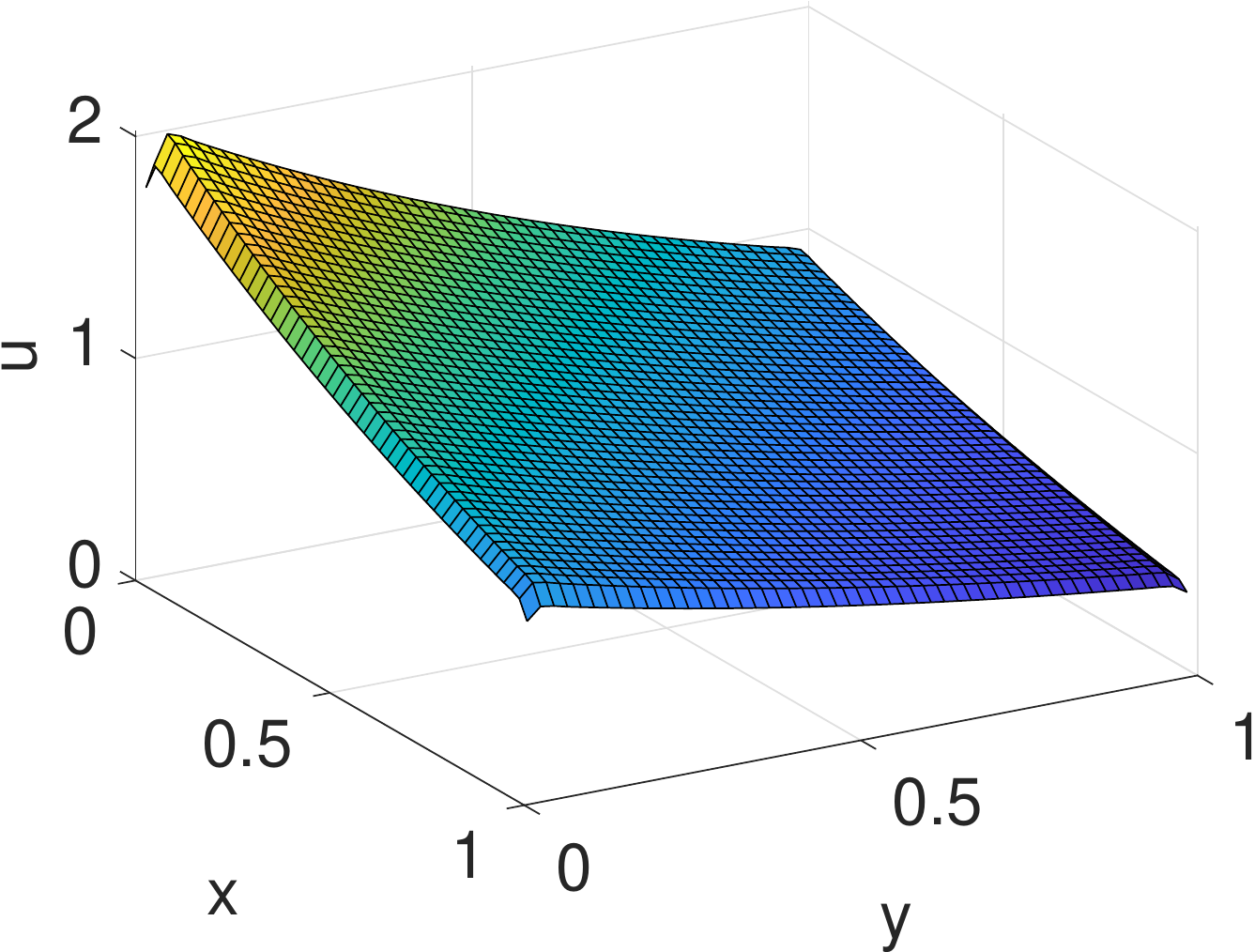}\\
  $\alpha=0.6$ &  $\alpha=0.8$ & $\alpha=1.0$
  \end{tabular}
  \caption{The probability density function for Example \ref{exam:2d} at $t=1$.\label{fig:density-2d-t=1}}
\end{figure}

\appendix
\section{Dykstra algorithm}\label{app:JKO-stepping}

In this appendix, we describe the Dykstra algorithm for JKO stepping, originally developed in \cite{Peyre:2015}.
Let $\mathbf{C}=[c_{ij}]\in\mathbb{R}^{M\times M}$ be the cost, with $c_{ij}=|x_i-x_j|^2$. The entropic regularization of
Wasserstein distance between two discrete probability measures $\mathbf{p},\mathbf{q}\in\Sigma_M$ for a cost $\mathbf{C}\in\mathbb{R}^{M\times M}$ is given by
\begin{equation*}
  W_{2,\gamma}(\mathbf{p},\mathbf{q})^2 = \min_{\boldsymbol{\pi} \in \mathcal{P}(\mathbf{p},\mathbf{q})} \langle \mathbf{C},\boldsymbol{\pi}\rangle + \gamma \langle
  \boldsymbol{\pi}, \log\boldsymbol{\pi}-\mathbf{1}\rangle + \langle \mathbf{1}, \ell_{\mathbb{R}_+^{M\times M}}(\boldsymbol{\pi})\rangle,
\end{equation*}
where $\gamma>0$ is a small number, controlling the tradeoff between accuracy and computational efficiency, and $\mathcal{P}(\mathbf{p},\mathbf{q})$ is
the set of couplings between $\mathbf{p}$ and $\mathbf{q}$, i.e., $\mathcal{P}(\mathbf{p},\mathbf{q})=\{\boldsymbol{\pi}\in\mathbb{R}^{M\times M}_+: \boldsymbol{\pi}\mathbf{1}= \mathbf{p}, \boldsymbol{\pi}^T\mathbf{1}=\mathbf{q}\}$,
with $\mathbf{1}$ being a vector or a matrix with all entries equal to unit. Accordingly, the entropic regularization of the fractional
JKO functional is given by
\begin{equation}\label{eqn:JKO-reg}
    \langle \mathbf{C}, \boldsymbol{\pi}\rangle + \gamma \langle \boldsymbol{\pi},\log \boldsymbol{\pi}-\boldsymbol{1}\rangle  + \langle\mathbf{1}, \ell_{\mathbb{R}_+^{M\times M}}(\boldsymbol{\pi})\rangle + \tau' f(\boldsymbol{\pi}\mathbf{1}) + \ell_{\mathcal{C}_q}(\boldsymbol{\pi}),
\end{equation}
where $\ell_C$ is an indicator function, $\tau'=\frac{2\tau^\alpha}{C_\alpha}$, $\mathcal{C}_q = \{\boldsymbol{\pi}\in
\mathbb{R}^{M\times M}: \boldsymbol{\pi}^T \mathbf{1} = q\}$ and $f(\mathbf{q})=\langle \mathbf{q}, \log\mathbf{q}-\mathbf{1}
+\boldsymbol{\psi}\rangle$. This functional can be recast into
\begin{equation*}
  \min_{\boldsymbol{\pi}} {\rm KL}(\boldsymbol{\pi}|\boldsymbol{\xi}) + \varphi_1(\boldsymbol{\pi}) + \varphi_2(\boldsymbol{\pi}),
\end{equation*}
with $\mathrm{KL}$ being the classical KL divergence, and
\begin{equation*}
  \varphi_1(\boldsymbol{\pi}) = \ell_{\mathcal{C}_q}(\boldsymbol{\pi}) \quad\mbox{and} \quad \varphi_2(\boldsymbol{\pi}) = \tfrac{\tau'}{\gamma} f(\boldsymbol{\pi}\mathbf{1}),
\end{equation*}
and the Gibbs kernel $\boldsymbol{\xi}$ is given by
\begin{equation*}
   \boldsymbol{\xi}= e^{-\mathbf{C}/\gamma} \in\mathbb{R}^{M\times M}_{+,*}.
\end{equation*}
The update is obtained using $\mathbf{p}=\boldsymbol{\pi}\mathbf{1}$. It remains to minimize \eqref{eqn:JKO-reg} with respect to the
coupling $\boldsymbol{\pi}\in\mathcal{P}(\mathbf{p},\mathbf{q})$. This can be carried out using the Dykstra algorithm developed in
\cite{Peyre:2015}; see Algorithm \ref{alg:Dykstra} for the complete procedure, where the notation $\circ$ denotes componentwise
product between two vectors. It is noteworthy that the algorithm operates only on
vectors $\mathbf{a},\mathbf{b},\mathbf{u},\mathbf{v}$ instead of the coupling $\boldsymbol{\xi}$ directly, due to the fact that
the optimal coupling satisfies $\boldsymbol{\pi}=\mathrm{diag}(\mathbf{a})\boldsymbol{\xi}\mathrm{diag}(\mathbf{b})$, for some
$\mathbf{a},\mathbf{b}\in \mathbb{R}_+^M$, like the classical entropic regularization of optimal transport \cite{Cuturi:2013}.

\begin{algorithm}
  \caption{Dykstra algorithm for JKO stepping.\label{alg:Dykstra}}
  \begin{algorithmic}[1]
    \STATE Set $\mathbf{a}^{0}=\mathbf{b}^{0}=\mathbf{u}^{0}=\mathbf{v}^{0}=\mathbf{1}$, and specify the tolerance $\epsilon$.
    \FOR{$\ell=1,\ldots,L$}
      \IF{$\ell$ is odd}
        \STATE update $\mathbf{a}^{\ell}$ and $\mathbf{b}^{\ell}$ by
         \begin{equation}
           \mathbf{a}^{\ell} = \mathbf{a}^{\ell-1}\circ \mathbf{u}^{\ell-2}\quad \mbox{and}\quad \mathbf{b}^{\ell} = \frac{\mathbf{q}}{\boldsymbol{\xi}^T(\mathbf{a}^{\ell})};
         \end{equation}
      \ELSE
        \STATE update $\mathbf{a}^{\ell}$ and $\mathbf{b}^{\ell}$ by
        \begin{equation}
          \mathbf{b}^{\ell}=\mathbf{b}^{\ell-1}\circ \mathbf{v}^{\ell-2} \quad \mbox{and}\quad \mathbf{a}^{\ell} = \frac{\mathbf{p}^{\ell}}{\boldsymbol{\xi}(\mathbf{b}^{\ell})},
        \end{equation}
        where $\mathbf{p}^{\ell}$ is given by
        \begin{equation}\label{eqn:KL-prox}
         \mathbf{p}^{\ell} = \mathrm{Prox}_{\frac\tau\gamma f}^{\rm KL}(\mathbf{a}^{\ell-1}\circ \mathbf{u}^{\ell-2}\circ \boldsymbol{\xi}(\mathbf{b}^{\ell}));
        \end{equation}
      \ENDIF
      \STATE update $\mathbf{u}^{\ell}$ and $\mathbf{v}^{\ell}$ by
      \begin{equation*}
        \mathbf{u}^{\ell}=\mathbf{u}^{\ell-2}\circ \frac{\mathbf{a}^{\ell-1}}{\mathbf{a}^{\ell}} \quad \mbox{and} \quad \mathbf{v}^{\ell} = \mathbf{v}^{\ell-2}\circ\frac{\mathbf{b}^{\ell-1}}{\mathbf{b}^{\ell}};
      \end{equation*}
      \STATE if $\|\mathbf{b}^\ell\circ \boldsymbol{\xi}^T(\mathbf{a}^{\ell})-\mathbf{q}\|<\epsilon$ and $\ell$ is even, terminate the iteration;
    \ENDFOR
    \STATE Output $\mathbf{p}^{\ell} $ defined in \eqref{eqn:KL-prox}.
  \end{algorithmic}
\end{algorithm}

The (Kullback-Leibler) KL proximal operator $ \mathrm{Prox}_{\sigma f}^{\rm KL}(\mathbf{q})$ in \eqref{eqn:KL-prox} for any $\mathbf{q}\in\mathbb{R}^M_+$ is defined by
\begin{equation*}
  {\rm Prox}_{\sigma f}^{\rm KL}(\mathbf{q}) = \arg\min _{\mathbf{p}\in\mathbb{R}^M_+} \langle \mathbf{p}, \log \tfrac{\mathbf{p}}{\mathbf{q}}- \mathbf{1} \rangle  + \sigma \langle  \mathbf{p}, \log\mathbf{p}-\mathbf{1}+\boldsymbol{\psi}\rangle.
\end{equation*}
Due to the separability of the optimization problem, it suffices to minimize the one-dimensional function
$g(s) =  s \log \frac{s}{t} - s + t  + \sigma (s \log s - s + s \psi).$ Differentiating with respect to
$s$ and setting it to zero gives $\log s - \log t + \sigma(\log s + \psi) =0$, i.e., $\log s = \frac{1}{1+\sigma}
\log t - \frac{\sigma \psi}{1+\sigma}$, and $s^*= t^\frac{1}{1+\sigma}e^{-\frac{\sigma}{1+\sigma}\psi}$.
Thus the proximal operator is given by
\begin{equation*}
  {\rm Prox}_{\sigma f}^{\rm KL}(\mathbf{q}) = \mathbf{q}^\frac{1}{1+\sigma}\circ e^{-\frac{\sigma}{1+\sigma}\boldsymbol{\psi}}.
\end{equation*}
The stopping criterion at line 9 employs the violation of the constraint $\mathcal{C}_q$.

\bibliographystyle{abbrv}

\end{document}